\documentclass[11pt]{amsart}

\usepackage{amsmath,amssymb,amsthm}
\usepackage[margin=1.08in]{geometry}
\usepackage[pdftex,colorlinks,citecolor=black,linkcolor=black,urlcolor=black,
bookmarks=true]{hyperref}
\usepackage{microtype}
\usepackage{setspace}

\newtheorem{theorem}{Theorem}[section]
\newtheorem{proposition}[theorem]{Proposition}
\newtheorem{lemma}[theorem]{Lemma}
\newtheorem{corollary}[theorem]{Corollary}
\numberwithin{equation}{section}
\theoremstyle{definition}
\newtheorem{definition}[theorem]{Definition}
\newtheorem{example}[theorem]{Example}
\theoremstyle{remark}
\newtheorem{remark}[theorem]{Remark}

\title[The Delsarte bound for codes and spheres]
{The Rate of the Delsarte Bound for Binary Codes and Spheres}
\author{Andrew Salmon}
\address{Independent researcher}
\date{October 13, 2026}

\begin{document}

\begin{abstract}
  We give the asymptotic rate of Delsarte's linear programming bound for binary codes and spherical codes.  This extends previous work of the author which formulated and proved an analogue of the conjecture of Afkhami-Jeddi, Cohn, Hartman, de Laat, and Tajdini for binary codes, obtaining the rate of the Delsarte linear program at distance $1/2-1/\pi$.

  The lower bounds introduce ideas from potential theory into the study of Delsarte's bounds, whereby the measures underlying a system of orthogonal polynomials give an obstacle problem.  We solve the resulting obstacle problem in the case of binary codes and spherical codes, obtaining a closed form for the lower bound in terms of elliptic integrals.  For the upper bounds, we relate the $N$-qubit construction for binary codes used in previous work of the author to the constructions in OpenAI's Chapter 2 for spherical codes and observe that they both involve interlacing spectra.  This leads to a variational problem involving measures constructed from these spectra under the Markov-Krein correspondence.  This variational problem admits an exact solution and matches the lower bound.
\end{abstract}

\maketitle
\begingroup
\small
\tableofcontents
\endgroup

\section{Introduction}

Binary codes are subsets of $\{0,1\}^n$ whose distinct words have Hamming
distance at least $d\in\{1,\ldots,n\}$.
Spherical codes are subsets of $S^{n-1}$ whose distinct points have inner
product at most $s\in[-1,1)$.
Delsarte's linear programming bounds for binary codes~\cite{Delsarte1973} and spherical
codes~\cite{DGS1977} give upper bounds on the size of the largest codebook; see
\cite{DelsarteLevenshtein1998} for the association-scheme framework.
For decades, the best general asymptotic (high-dimensional) upper bounds were due to
McEliece, Rodemich, Rumsey, and Welch~\cite{MRRW1977} and Kabatiansky
and Levenshtein~\cite{KL1978}, respectively, which arise from feasible
solutions to the Delsarte linear program.
Recently, Chapter~2 of OpenAI's \emph{Ten Advances in Mathematics and
Theoretical Computer Science} improved both bounds~\cite{OpenAI2026}.
Related binary coding upper bounds were developed using classical-quantum
channels by Alrabiah and Guruswami~\cite{AG2026}.

The asymptotic limit of linear programming bounds has been the subject of study by many authors over the years since the classic work from the 1970s, and they give the best-known asymptotic bounds for both binary codes and spherical codes.
Barg and Jaffe~\cite{BargJaffe2001} studied the binary
program numerically at block length $1000$, testing the closeness of the
known bounds.  Kalai asked about the asymptotic value~\cite[Questions~3--4]{Kalai2024}.
Samorodnitsky~\cite{Samorodnitsky2001} and Navon and
Samorodnitsky~\cite{NavonSamorodnitsky2005} established lower bounds on its
binary asymptotic rate, and Samorodnitsky~\cite{Samorodnitsky2004} treated the
spherical case, while Torquato and Stillinger~\cite{TorquatoStillinger2006}
adapted Samorodnitsky's method to the sphere packing context.
Afkhami-Jeddi, Cohn, Hartman, de Laat, and Tajdini~\cite{ACHT2020}
studied the sphere-packing linear programming bound numerically via the
modular bootstrap, with extrapolations indicating an exponential improvement
over the Kabatiansky--Levenshtein bound.
They conjectured that the optimal bound for sphere-packing implied by the Cohn--Elkies linear program in dimension $n$ is $(e/(2\pi))^{n/2+o(n)}$~\cite[Conjecture~3.2]{ACHT2020}.  This was subsequently proved in~\cite[Chapter~1, Theorem~1.1]{OpenAI2026}.
Here, we determine the exact asymptotic rates of the
binary and spherical Delsarte programs.

Write $D_2(n,d)$ and $D_{\mathbb S}(n,s)$ for the optima of the binary and
spherical Delsarte programs defined in Section~\ref{sec:programs}.  These linear programs
bound the sizes of the corresponding codes.  For $0<\delta\leq1$ and
$-1\leq s<1$, put
\begin{align*}
 \underline R_D(\delta)
   &=\liminf_{n\to\infty}\frac1n\log_2D_2(n,\lceil\delta n\rceil),&
 \overline R_D(\delta)
   &=\limsup_{n\to\infty}\frac1n\log_2D_2(n,\lceil\delta n\rceil),\\
 \underline R_D^{\mathbb S}(s)
   &=\liminf_{n\to\infty}\frac1n\ln D_{\mathbb S}(n,s),&
 \overline R_D^{\mathbb S}(s)
   &=\limsup_{n\to\infty}\frac1n\ln D_{\mathbb S}(n,s).
\end{align*}

The answer in both cases is parametrized by the complete elliptic integrals
of the first and second kinds,
\[
 K(\kappa)=\int_0^{\pi/2}\frac{dt}{\sqrt{1-\kappa^2\sin^2t}},
 \qquad
 E(\kappa)=\int_0^{\pi/2}\sqrt{1-\kappa^2\sin^2t}\,dt.
\]
For $0<\kappa<1$, put $\kappa'=\sqrt{1-\kappa^2}$,
$K=K(\kappa)$, $E=E(\kappa)$, $K'=K(\kappa')$, and $E'=E(\kappa')$;
the primes denote the complementary modulus.  Define
\[
 \delta_+(\kappa)=\frac12-\frac E\pi,
 \qquad
 \delta_-(\kappa)=\frac12-\frac{E-\kappa'^2K}{\pi\kappa},
 \qquad
 s(\kappa)=\frac{K-E}{\kappa K},
\]
and
\[
 B(\kappa)=\frac14\ln\kappa+\frac{\kappa'^2KK'}{2\pi}.
\]
The binary and spherical rate functions are the functions $\mathcal J$ and
$F_{\mathrm S}$ given by
\begin{align}
 \mathcal J\bigl(\delta_\pm(\kappa)\bigr)
   &=\frac12\pm\frac{B(\kappa)}{\ln2},\label{eq:binary-profile}\\
 F_{\mathrm S}\bigl(s(\kappa)\bigr)
   &=\frac12+\frac12\ln\frac K\pi+\frac14\ln\kappa
     +\frac{K\bigl((1+\kappa^2)K'-2E'\bigr)}{2\pi}.\label{eq:spherical-profile}
\end{align}
The branches $\delta_+$ and $\delta_-$ cover
$(0,\frac12-\frac1\pi)$ and $(\frac12-\frac1\pi,\frac12)$,
respectively, meeting at $\mathcal J(\frac12-\frac1\pi)=\frac12$.
The parameter $s(\kappa)$ covers $(0,1)$.

\begin{theorem}[Main result]\label{thm:main}
For $0<\delta<\frac12$ and $0<s<1$,
\begin{align*}
 \underline R_D(\delta)=\overline R_D(\delta)
   &=R_D(\delta)=\mathcal J(\delta),\\
 \underline R_D^{\mathbb S}(s)=\overline R_D^{\mathbb S}(s)
   &=R_D^{\mathbb S}(s)=F_{\mathrm S}(s).
\end{align*}
\end{theorem}

\begin{remark}[The limit of the channel bounds]\label{rem:barg-channel-limit}
Barg~\cite[Section~8.1.3]{Barg2026} asks for the limit of
\[
 \mathcal R_\ell(\delta)=\inf\{\chi(\sigma_0,\sigma_1):
 \sigma_0,\sigma_1\text{ are output-symmetric states on }\mathbb C^\ell,
 \ p_{\mathrm e}(\sigma_0,\sigma_1)<\delta\}.
\]
Here output symmetry means that a unitary exchanges the two density
matrices, $p_{\mathrm e}$ is the error probability of the pretty-good
measurement with equal priors, and
$\chi=S_2((\sigma_0+\sigma_1)/2)-(S_2(\sigma_0)+S_2(\sigma_1))/2$ is the Holevo information,
where $S_2(\sigma)=-\operatorname{tr}(\sigma\log_2\sigma)$.
These are the bounds from the pretty good criterion of
Alrabiah and Guruswami~\cite[Theorem~1]{AG2026}.
Theorem~\ref{thm:main} and the constructions in Section~\ref{sec:upper-bounds}
answer Barg's question:
\[
 \lim_{\ell\to\infty}\mathcal R_\ell(\delta)
 =\mathcal J(\delta)=R_D(\delta),\qquad 0<\delta<\tfrac12.
\]
To see this, a classical-quantum channel with Holevo information $\chi$ and decoding error $p_{\mathrm e}$ is, in the language used by~\cite{Salmon2026HalfRate}, a normalized Horn pair, which can be asymptotically realized as feasible solutions of the Delsarte program~\cite[Theorem~2.20]{Salmon2026HalfRate}.
Therefore, the pretty good criterion's bound~\cite[Theorem~1]{AG2026} is never stronger than Delsarte, as observed in~\cite[Section~4.1]{Salmon2026HalfRate}.
Conversely, the proof of Theorem~\ref{thm:horn-upper} shows that the $N$-qubit construction, which is a particular choice of classical-quantum channel constructed in~\cite[Section~4.2]{Salmon2026HalfRate} achieves the optimal rate of Delsarte's program.
\end{remark}

\begin{remark}[Earlier results at special values]\label{rem:earlier-results}
At $\delta_\pi=\frac12-\frac1\pi$, Theorem~\ref{thm:main} gives
$R_D(\delta_\pi)=\frac12$, the main result of
\cite{Salmon2026HalfRate}.  In the spherical case, the closed form gives
\[
 F_{\mathrm S}(s)
 =\frac12\ln\frac{e}{\pi(1-s)}+o(1)
 \qquad(s\uparrow1).
\]
This is the spherical limit corresponding to the Cohn--Elkies~\cite{CohnElkies2003}
sphere-packing exponent in \cite[Chapter~1]{OpenAI2026}.
Chapter~2 of the same work recovers the Cohn-Elkies upper bound from the moving-subspace construction~\cite[Chapter~2, Theorem~8.3]{OpenAI2026}.  Here, we show that one can also recover the lower bound as well.
\end{remark}

\begin{remark}[Kissing numbers]\label{rem:kissing}
Put $r_{\mathrm{LP}}=R_D^{\mathbb S}(\frac12)$, in nats per
dimension.  At the kissing angle $s=\frac12$, the bounds on the
Delsarte rate are
\begin{center}
\begin{tabular}{@{}ll@{}}
Source & Bound on $r_{\mathrm{LP}}$ \\
\hline
Kabatiansky--Levenshtein~\cite{KL1978}
  & $r_{\mathrm{LP}}\leq 0.277913\ldots$ \\
OpenAI, Chapter~2, level two~\cite{OpenAI2026}
  & $r_{\mathrm{LP}}\leq0.27491$ \\
Theorem~\ref{thm:main}
  & $r_{\mathrm{LP}}=0.2749027991\ldots$ \\
Combination of~\cite{CohnZhao2014,OpenAI2026}
  & $r_{\mathrm{LP}}\geq\frac12\ln\frac{2e}{\pi}
     =0.2742086473\ldots$ \\
Samorodnitsky~\cite[Corollary~1.3]{Samorodnitsky2004}
  & $r_{\mathrm{LP}}\geq 0.211039\ldots$
\end{tabular}
\end{center}
Theorem~\ref{thm:main} gives the exponent of the Delsarte bound on the
kissing number in $\mathbb R^n$.  The $0.274208\ldots$ lower bound is the Cohn--Zhao transfer~\cite[Theorem~3.4]{CohnZhao2014}
of the Cohn--Elkies LP exponent in~\cite[Chapter~1]{OpenAI2026}.

Kominers~\cite[Theorem~1 and Corollary~4]{Kominers2026} obtains the upper
exponent $\frac12\ln(2e/\pi)=0.2742086473\ldots$ for lattice kissing
numbers and for the average contact degree of finite packings of congruent
balls, using the auxiliary functions of~\cite[Chapter~1]{OpenAI2026}.
Thus the Kominers bound has a strictly smaller exponent than the Delsarte bound.
\end{remark}

\subsection*{Numerical comparisons}

Tables~\ref{tab:binary-rates} and~\ref{tab:spherical-rates} compare the
rates of the Delsarte linear programs and compare them with earlier bounds above and below, in bits per coordinate for binary
codes and nats per dimension for spherical codes, rounded to eight decimal
places.  The finite-level upper bounds are evaluated
numerically.  Code and numerical data reproducing both tables are included
as ancillary material with the arXiv submission.
The Gilbert--Varshamov and Shannon rows give achievable code
rates, while the Samorodnitsky and Navon--Samorodnitsky rows give lower bounds on
the Delsarte LP.

For binary codes, the OpenAI row takes the minimum of the whole-cube and
constant-weight bounds of~\cite[Chapter~2]{OpenAI2026}; the GJL row gives
the first honeycomb level of~\cite{GJL2026} on its own.
The remaining bounds are from~\cite{MRRW1977,Samorodnitsky2001,NavonSamorodnitsky2005}.  The spherical lower bound is
\cite[Corollary~1.3]{Samorodnitsky2004}.

\begin{table}[htbp]
\centering
\footnotesize
\setlength{\tabcolsep}{3pt}
\begin{tabular}{@{}lrrrrrr@{}}
$\delta$ & $0.05$ & $0.10$ & $\frac12-\frac1\pi$ & $0.25$ & $0.35$ & $0.45$ \\
\hline
MRRW2 & 0.82513681 & 0.69274074 & 0.50177980 & 0.35371054 & 0.15813294 & 0.02526613 \\
OpenAI & 0.82511480 & 0.69255746 & 0.50055474 & 0.35037908 & 0.15750558 & 0.02525738 \\
GJL, first level & 0.82513681 & 0.69274074 & 0.50034777 & 0.34963011 & 0.15745771 & 0.02525727 \\
\hline
Exact LP rate & 0.82511283 & 0.69252461 & 0.50000000 & 0.34959895 & 0.15745731 & 0.02525727 \\
\hline
Samorodnitsky & 0.78591946 & 0.62646625 & 0.41465946 & 0.27165039 & 0.11203244 & 0.01624584 \\
Navon--Samorodnitsky & 0.49405402 & 0.48547530 & 0.38797620 & 0.28411932 & 0.13471088 & 0.02275530 \\
Gilbert--Varshamov & 0.71360304 & 0.53100441 & 0.31623954 & 0.18872188 & 0.06593194 & 0.00722555
\end{tabular}
\caption{Binary rates in bits per coordinate.  The exact LP rate is
$R_D(\delta)$ from Theorem~\ref{thm:main}.}
\label{tab:binary-rates}
\end{table}

\begin{table}[htbp]
\centering
\footnotesize
\setlength{\tabcolsep}{3pt}
\begin{tabular}{@{}lrrrrrr@{}}
$s$ & $0.10$ & $0.25$ & $0.50$ & $0.75$ & $0.90$ & $0.95$ \\
\hline
Kabatiansky--Levenshtein & 0.01759504 & 0.08393636 & 0.27791336 & 0.62448696 & 1.08263232 & 1.42920591 \\
OpenAI, first level & 0.01758891 & 0.08370759 & 0.27492039 & 0.62092418 & 1.07906955 & 1.42564314 \\
\hline
Exact LP rate & 0.01758891 & 0.08370753 & 0.27490280 & 0.62078246 & 1.07892760 & 1.42550119 \\
\hline
Samorodnitsky & 0.01131011 & 0.05810281 & 0.21103985 & 0.52416443 & 0.96696417 & 1.30900350 \\
Shannon & 0.00502517 & 0.03226926 & 0.14384104 & 0.41333929 & 0.83036560 & 1.16395145
\end{tabular}
\caption{Spherical rates in nats per dimension.  The exact LP rate is
$R_D^{\mathbb S}(s)$ from Theorem~\ref{thm:main}.}
\label{tab:spherical-rates}
\end{table}

\subsection*{Use of Generative Artificial Intelligence}

The author and GPT-5.6 Sol formulated the upper bounds for spheres and
binary codes during the writing of~\cite{Salmon2026HalfRate}.
It seemed possible that these bounds might be sharp, although the author was initially discouraged by the complicated-looking measures $\mathfrak{m}$ having a less ``canonical'' looking structure than the arcsine law that appears at $1/2-1/\pi$.  However, upon exploration, several
structural features supported the possibility of a sharp bound: the spherical and binary
constructions have the same form, with bosonic and fermionic realizations,
respectively, and there were two special points at which sharpness was
known.  Moreover, the Cohn--Zhao construction gives a rate close to the
optimal one while being provably suboptimal at large angles.

The author initially worked with Sol and Astra models on incremental
strengthenings of the Navon--Samorodnitsky argument~\cite{NavonSamorodnitsky2005}.
Several variants of the Navon--Samorodnitsky idea for both binary codes and spherical codes improve the large-distance region, and it seems plausible that a sufficiently powerful and clever version of those arguments could be extended to give the lower bound result.
The argument here, however, originated from the
gradual realization (on a human timescale) that in the binary
case, given the upper bounds, it suffices to prove
\[
 R_D(\delta_+(\kappa))+R_D(\delta_-(\kappa))\ge1,
\]
with the explicit relation between the two sides of $1/2 - 1/\pi$
encoded by $\delta_+(\kappa)$ and $\delta_-(\kappa)$, and that this could
be approached by extending the techniques of~\cite[Section~3]{Salmon2026HalfRate}.

Opus 5.5 identified the compatibility of the MacWilliams transform with the predicted relation.
This suggests two natural variants of the half-disk Dirichlet problem in which we replace the imaginary side with a shorter segment or replace the outer circle with two arcs.
The relevant Riemann surface gluing along these cuts is genus 1, which gives hope for an exact connection via elliptic functions.
From there, Opus, with light direction from the author,
recast the half-rate result as a thin obstacle problem, showing that a natural ``tilted mass concentration'' idea could be made to work for arbitrary distances, and in doing so identified the right generalization of~\cite[Lemma~3.3]{Salmon2026HalfRate}.

The draft was prepared with Codex.

\section{The Delsarte programs}\label{sec:programs}

Let $G$ be a compact group, $H$ a closed subgroup, and $X=G/H$,
with invariant probability measure $\mu_X$.
Let $\tau:X\times X\to[-1,1]$ be a continuous, symmetric,
$G$-invariant relation with $\tau(x,x)=1$, whose value determines the
$G$-orbit of an ordered pair.  Transitivity implies that the distribution
$\nu$ of $\tau(x,y)$ for $y\sim\mu_X$ is independent of the fixed point $x$.

For $s<1$, the \emph{threshold graph} $\mathcal G_X(s)$ has vertex set
$X$ and joins distinct vertices $x,y$ when $\tau(x,y)>s$.  Its independent
sets are the subsets whose distinct pair relations are at most
$s$.  Write $\alpha(\mathcal G_X(s))$ for the largest size of such a set.

The polynomial kernels come from orthogonal projections.  In the binary
and spherical cases below,
$L^2(X,\mu_X)$ is an orthogonal sum of finite-dimensional $G$-invariant
spaces $\mathcal H_j$, indexed by degrees $j$, with
$\mathcal H_0$ the constants.  Put $d_j=\dim\mathcal H_j$ and choose an
orthonormal basis $e_{j,1},\ldots,e_{j,d_j}$.  The
kernel
\begin{equation}\label{eq:homogeneous-projection-kernel}
 p_j(\tau(x,y))=
 \frac1{d_j}\sum_{r=1}^{d_j}e_{j,r}(x)\overline{e_{j,r}(y)}
\end{equation}
is independent of the basis and invariant under the diagonal
$G$-action, hence depends only on $\tau(x,y)$.  Its diagonal is constant
by transitivity; integrating the diagonal shows that the normalization
in \eqref{eq:homogeneous-projection-kernel} gives $p_j(1)=1$.
These are real polynomials in the two cases we consider and are the normalized
Krawtchouk and Gegenbauer polynomials displayed below.  Orthogonality of
the projection spaces gives their orthogonality for $\nu$, and $p_0=1$.

Positive semidefiniteness is immediate:
\eqref{eq:homogeneous-projection-kernel} gives the Mercer decomposition.  Moreover, a polynomial is positive semidefinite as a kernel if and only if its Fourier coefficients are nonnegative, by Bochner's theorem.

\begin{definition}[Delsarte program]\label{def:polynomial-program}
For the threshold graph $\mathcal G_X(s)$, an admissible auxiliary
function is
\[
  f(t)=\sum_{j}a_jp_j(t),\qquad
  a_0=1,\quad a_j\geq0.
\]
The coefficient sum is finite.  It is feasible at threshold $s$ if
$f(t)\leq0$ for $t\in\operatorname{supp}\nu\cap[-1,s]$.  The Delsarte value
of the threshold graph is $D_X(s)=\inf f(1)$ over feasible $f$.  The positive-kernel
bound~\cite{Delsarte1973,DGS1977} gives
\[
  \alpha(\mathcal G_X(s))\leq D_X(s).
\]
\end{definition}

\begin{definition}[Binary Hamming scheme]\label{def:binary-program}
Let $n\geq1$ and $H_n=\{0,1\}^n$, with transitive group
$(\mathbb Z/2\mathbb Z)^n\rtimes S_n$ and uniform probability measure.
If $d_H(x,y)$ is Hamming distance,
put $\tau(x,y)=1-2d_H(x,y)/n$.  The relation distribution gives mass
$2^{-n}\binom ni$ to $1-2i/n$, for $0\leq i\leq n$.  Its orthogonal
polynomials are the normalized Krawtchouk polynomials
\[
  B_j^{(n)}(1-2i/n)=\frac{K_j^{(n)}(i)}{\binom nj}
  \qquad(0\leq i,j\leq n),
\]
where
\[
  K_j^{(n)}(i)=\sum_{k=0}^{j}(-1)^k\binom{i}{k}\binom{n-i}{j-k}
  \qquad(0\leq i,j\leq n).
\]
For minimum distance $1\leq d\leq n$, put $s_{n,d}=1-2d/n$.  Then
$\mathcal G_{H_n}(s_{n,d})$ joins words whose Hamming distance is less
than $d$.  Thus the binary Delsarte value is
\[
  D_2(n,d)=D_{H_n}(s_{n,d}).
\]
If $A_2(n,d)$ is the largest size of a binary code of length $n$ and
minimum distance $d$, then
$A_2(n,d)=\alpha(\mathcal G_{H_n}(s_{n,d}))\leq D_2(n,d)$.
\end{definition}

\begin{definition}[Sphere]\label{def:spherical-program}
Let $n\geq3$ and $S^{n-1}\subset\mathbb R^n$ be the unit sphere,
with transitive group $O(n)$ and normalized surface measure.  Here
$\tau(x,y)=\langle x,y\rangle$, and the relation distribution has density proportional to
$(1-t^2)^{(n-3)/2}$ on $[-1,1]$.  If
$P_j^{((n-3)/2,(n-3)/2)}$ is the Jacobi polynomial of degree $j$, the
normalized orthogonal polynomial is
\[
  G_j^{(n)}(t)=
  \frac{P_j^{((n-3)/2,(n-3)/2)}(t)}
       {P_j^{((n-3)/2,(n-3)/2)}(1)}.
\]
These are the normalized Gegenbauer polynomials.  For $-1\leq s<1$,
$\mathcal G_{S^{n-1}}(s)$ joins distinct points with inner product greater
than $s$, and the spherical Delsarte value is
\[
  D_{\mathbb S}(n,s)=D_{S^{n-1}}(s).
\]
If $A_{\mathbb S}(n,s)$ is the largest size of a subset of $S^{n-1}$
whose distinct inner products are at most $s$, then
$A_{\mathbb S}(n,s)=\alpha(\mathcal G_{S^{n-1}}(s))\leq D_{\mathbb S}(n,s)$.
\end{definition}

From now on, we work in either the binary Hamming scheme or the sphere
defined above.  The relation distribution $\nu$ is the orthogonality measure.

\begin{definition}[Support function]\label{def:support-function}
For $\lambda=\sigma+\mathrm i\eta\in\mathbb C$, the support function of
the orthogonality measure is
\[
 \mathsf h(\lambda)=\sup_{t\in\operatorname{supp}\nu}\Re(\lambda t)
   =|\sigma|.
\]
The last equality holds in both cases because the support is contained
in $[-1,1]$ and contains both endpoints.
\end{definition}

\begin{definition}\label{def:normalized-log-modulus}
Let $g$ be integrable with respect to $\nu$, and suppose
\[
 M(g)=\int|g(t)|\,d\nu(t)>0.
\]
Put
\[
 \widehat g(\lambda)=\int g(t)e^{n\lambda t}\,d\nu(t).
\]
This function is entire: since $\nu$ is supported on $[-1,1]$ and
$g$ is integrable, differentiation under the integral is justified
locally uniformly in $\lambda$.  Define its normalized log-modulus by
\[
 v_g(\lambda)=\frac1n\ln\frac{|\widehat g(\lambda)|}{M(g)}.
\]
We use the convention $\ln0=-\infty$.
The log-modulus of an entire function is subharmonic, so $v_g$ is
subharmonic on $\mathbb C$.
\end{definition}

\begin{lemma}\label{lem:support-bound}
For every such $g$ and every $\lambda\in\mathbb C$,
\[
 v_g(\lambda)\leq\mathsf h(\lambda).
\]
In particular, $v_g(\mathrm i\eta)\leq0$ for every $\eta\in\mathbb R$.
\end{lemma}

\begin{proof}
The triangle inequality and the definition of the support function give
\[
 |\widehat g(\lambda)|
 \leq\int |g(t)|e^{n\Re(\lambda t)}\,d\nu(t)
 \leq M(g)e^{n\mathsf h(\lambda)}.
\]
Divide by $M(g)$ and take logarithms.  The last assertion
follows from $\mathsf h(\mathrm i\eta)=0$.
\end{proof}

For the corresponding normalized orthogonal polynomials, define the
Laplace transforms of the signed measures $p_j\,d\nu$, evaluated at
$n\lambda$, by
\[
 \beta_j(\lambda)=\int p_j(t)e^{n\lambda t}\,d\nu(t),
 \qquad
 \Xi_n(\lambda)=\frac1n\ln\sup_j|\beta_j(\lambda)|.
\]
The following proposition shows, among other things, that $\Xi_n$ is
well-defined.

\begin{proposition}\label{lem:polynomial-support-bound}
For every degree $j$ and every $\lambda\in\mathbb C$,
\[
 |\beta_j(\lambda)|\leq e^{n\mathsf h(\lambda)}.
\]
Consequently, $\Xi_n(\lambda)\leq\mathsf h(\lambda)$.
\end{proposition}

\begin{proof}
Every $t\in\operatorname{supp}\nu$ is realized as $\tau(x,y)$ in the
binary and spherical settings.  The positive semidefinite kernel
$p_j(\tau(x,y))$ has diagonal entries $p_j(1)=1$, so its two-point
principal matrix gives $|p_j(t)|\leq1$.  Since $\nu$ is a probability
measure, the triangle inequality gives
\[
 |\beta_j(\lambda)|
 \leq\int |p_j(t)|e^{n\Re(\lambda t)}\,d\nu(t)
 \leq e^{n\mathsf h(\lambda)}.
\]
Taking the supremum over $j$ proves the bound on $\Xi_n$.
\end{proof}

The function $\Xi_n$ is continuous and subharmonic.  Indeed, orthogonality
and $|p_j|\le1$ give
$|\beta_j(\lambda)|\le\sum_{k\ge j}(n|\lambda|)^k/k!$,
so the tail tends to zero locally uniformly.  The transforms cannot all
vanish at one point: their vanishing would imply that integration against
$e^{n\lambda t}d\nu(t)$ annihilates every polynomial, and hence, by
polynomial approximation, also $e^{-n\lambda t}$, a contradiction.
Thus locally their supremum is a positive maximum of finitely many
moduli.  Its logarithm is continuous and subharmonic.

\begin{example}[Binary codes]\label{ex:binary-transforms}
For the binomial measure and normalized Krawtchouk polynomials,
\begin{equation}\label{eq:binary-transforms}
 \beta_j(\lambda)=(\cosh\lambda)^{n-j}(\sinh\lambda)^j
 \qquad(0\leq j\leq n).
\end{equation}
This factorization reflects the independent coordinates of the binary
cube: the exponential contributes $\cosh\lambda$ on each coordinate,
and a coordinate appearing in a character contributes $\sinh\lambda$.
Consequently,
\begin{equation}\label{eq:binary-Xi}
 \Xi_n(\lambda)=\ln\max\{|\cosh\lambda|,|\sinh\lambda|\}.
\end{equation}
The supremum is attained at $j=0$ or $j=n$, and the expression is
independent of $n$.
\end{example}

\begin{example}[Spheres]\label{ex:spherical-transforms}
Put $\alpha_n=(n-2)/2$.  For the spherical measure and normalized
Gegenbauer polynomials, the classical Gegenbauer integral gives
\begin{equation}\label{eq:spherical-transforms}
 \beta_j(\lambda)=\Gamma(\alpha_n+1)
   \left(\frac2{n\lambda}\right)^{\alpha_n}
   I_{j+\alpha_n}(n\lambda),
\end{equation}
where $I_a$ is the modified Bessel function of the first kind and
$\Gamma$ is the gamma function; see~\cite[\S18.17(v)]{DLMF} and
analytic continuation.  The product is interpreted by its entire
continuation, so there is no singularity at $\lambda=0$.
For $\lambda\ne0$,
\begin{equation}\label{eq:spherical-Xi}
 \Xi_n(\lambda)=\frac1n\ln\left[
 \Gamma(\alpha_n+1)
 \left(\frac2{n|\lambda|}\right)^{\alpha_n}
 \sup_{j\geq0}|I_{j+\alpha_n}(n\lambda)|\right].
\end{equation}
At the origin, orthogonality gives $\beta_0(0)=1$ and
$\beta_j(0)=0$ for $j\geq1$, so $\Xi_n(0)=0$.
For real $x\ne0$, monotonicity of $I_a(n|x|)$ in its nonnegative
order~\cite[\S10.37]{DLMF} shows that the supremum occurs at $j=0$:
\begin{equation}\label{eq:spherical-Xi-real}
 \Xi_n(x)=\frac1n\ln\left[
 \Gamma(\alpha_n+1)
 \left(\frac2{n|x|}\right)^{\alpha_n}
 I_{\alpha_n}(n|x|)\right].
\end{equation}
On the imaginary axis this gives
\begin{equation}\label{eq:spherical-axis-transforms}
 |\beta_j(\mathrm i\eta)|=\Gamma(\alpha_n+1)
   \left(\frac2{n|\eta|}\right)^{\alpha_n}
   |J_{j+\alpha_n}(n|\eta|)|\qquad(\eta\ne0),
\end{equation}
where $J_a$ is the Bessel function of the first kind.
Thus, for $\eta\ne0$,
\begin{equation}\label{eq:spherical-Xi-imaginary}
 \Xi_n(\mathrm i\eta)=\frac1n\ln\left[
 \Gamma(\alpha_n+1)
 \left(\frac2{n|\eta|}\right)^{\alpha_n}
 \sup_{j\geq0}|J_{j+\alpha_n}(n|\eta|)|\right].
\end{equation}
For complex arguments the maximizing order need not be $j=0$;
the supremum over orders is retained in \eqref{eq:spherical-Xi}.
\end{example}

For an admissible auxiliary function $g$, define
\[
 R(g)=\frac1n\ln\frac{g(1)}{M(g)}.
\]
Here $g(1)>0$ and $M(g)>0$ because the constant term of $g$ is $1$.
The natural logarithmic rate of the Delsarte bound obtained using $g$
as a feasible solution is $n^{-1}\ln g(1)$.  The quantity $R(g)$ is
its $L^1$-normalized form.  Since $M(g)\geq|\int g\,d\nu|=1$, we have
$R(g)\leq n^{-1}\ln g(1)$.

\begin{lemma}\label{lem:envelope-bound}
For every admissible auxiliary function $g$ and every $\lambda\in\mathbb C$,
\[
 v_g(\lambda)\leq\Xi_n(\lambda)+R(g).
\]
\end{lemma}

\begin{proof}
Write $g=\sum_j a_jp_j$, with $a_j\geq0$.  By linearity,
$\widehat g=\sum_j a_j\beta_j$.  Since $p_j(1)=1$, we have
$\sum_j a_j=g(1)$, and hence
\[
 |\widehat g(\lambda)|
 \leq\sum_j a_j|\beta_j(\lambda)|
 \leq g(1)\sup_j|\beta_j(\lambda)|.
\]
Divide by $M(g)$ and take normalized logarithms.
\end{proof}

\section{From an obstacle problem to Delsarte lower bounds}\label{sec:barrier}

Write
$u^*(\lambda)=\limsup_{\zeta\to\lambda}u(\zeta)$ for upper
semicontinuous envelope.

\subsection{A finite-dimensional lower bound}

For $R\geq0$, define the finite-dimensional obstacle and its Perron
solution by
\[
 \psi_{n,R}=\min\{\mathsf h,R+\Xi_n\},\qquad
 V_{n,R}=\left(\sup\{v\text{ subharmonic on }\mathbb C:
                         v\leq\psi_{n,R}\}\right)^*.
\]
The preceding bounds give $v_g\leq V_{n,R}$ whenever $R(g)\leq R$.
To treat the two arguments together, write
$d\nu=\omega_n\,d\rho_n$.  In the binary case, $\rho_n$ is counting
measure on $\{t_i=1-2i/n:0\le i\le n\}$ and
$\omega_n(t_i)=2^{-n}\binom ni$; in the spherical case, $d\rho_n=dt$
and $\omega_n=d\nu/dt$ on $[-1,1]$.
Let $L_n=\rho_n([-1,1])$, and use the Fourier domain and normalization
$c_n$ in the following table:
\[
\begin{array}{c|c|c|c|c}
 X&d\rho_n&\mathcal I_n&L_n&c_n\\ \hline
 H_n&\text{counting measure}&[0,\pi]&n+1&1/\pi\\[2pt]
 S^{n-1}&dt&\mathbb R&2&n/(2\pi)
\end{array}
\]
The binary Fourier domain is one period of the transform modulus; the
spherical Fourier domain is the whole line.  The constants $c_n$ are the
Plancherel normalizations for the frequency convention $e^{\mathrm i n\eta t}$.

For $\sigma>0$, define
\begin{equation}\label{eq:finite-fourier-factor}
 Q_{n,R}(\sigma)^2=L_nc_n\int_{\mathcal I_n}
              e^{2nV_{n,R}(\sigma+\mathrm i\eta)}\,d\eta.
\end{equation}
Here $L_nc_n=(n+1)/\pi$ for binary codes and $L_nc_n=n/\pi$ for spheres.
We use its normalized logarithm
\[
 q_{n,R}(\sigma)=\frac1n\ln Q_{n,R}(\sigma),
\]
allowing the value $+\infty$.

\begin{lemma}[Tilted mass concentration]\label{lem:mass-concentration}
At fixed dimension, for every admissible $g$ with $R(g)\leq R$,
every $\sigma>0$, and every $-1<b<1$,
\begin{equation}\label{eq:finite-mass-concentration}
 \int_{t>b}|g(t)|\,d\nu(t)
 \leq M(g)e^{n(q_{n,R}(\sigma)-b\sigma)}.
\end{equation}
\end{lemma}

This lemma has the form of Chernoff's bound, although $M$ here is now a mass.
We give a direct proof below.

\begin{proof}
If $Q_{n,R}(\sigma)=+\infty$, the estimate is immediate.  Otherwise,
form the density of the tilted signed measure with respect to $\rho_n$:
\[
 r_\sigma(t)=e^{n\sigma t}g(t)\omega_n(t),\qquad
 \widehat g(\sigma+\mathrm i\eta)
 =\int r_\sigma(t)e^{\mathrm i n\eta t}\,d\rho_n(t).
\]
In the spherical case, extend $r_\sigma$ by zero outside $[-1,1]$.
The bound $v_g\leq V_{n,R}$ and finiteness of $Q_{n,R}$ make the
transform square-integrable.  Plancherel's theorem then gives
\begin{equation}\label{eq:shared-plancherel}
 \|r_\sigma\|_{L^2(\rho_n)}^2
 =c_n\int_{\mathcal I_n}|\widehat g(\sigma+\mathrm i\eta)|^2\,d\eta.
\end{equation}
Cauchy--Schwarz, Plancherel, and the bound $v_g\leq V_{n,R}$ give the conclusion
\begin{align*}
 \int_{t>b}|g(t)|\,d\nu(t)
 &\le e^{-nb\sigma}\|r_\sigma\|_{L^1(\rho_n)}\\
 &\le e^{-nb\sigma}\sqrt{L_n}\|r_\sigma\|_{L^2(\rho_n)}\\
 &\le e^{-nb\sigma}M(g)
   \left(L_nc_n\int_{\mathcal I_n}
         e^{2nV_{n,R}(\sigma+\mathrm i\eta)}\,d\eta\right)^{1/2}\\
 &=M(g)e^{n(q_{n,R}(\sigma)-b\sigma)}.
\end{align*}
\end{proof}

The sign condition forces more than half of the absolute mass into
the cap $t>s$.

\begin{lemma}[Positive mass]\label{lem:positive-mass}
For an admissible $g$, write $g^+(t)=\max\{g(t),0\}$.  Then
\[
 \int g^+\,d\nu=\frac{M(g)+1}{2}>\frac{M(g)}2.
\]
If $g$ is feasible at threshold $s$, then $g^+$ vanishes on
$\operatorname{supp}\nu\cap[-1,s]$.
\end{lemma}
\begin{proof}
Orthogonality and the normalization of the constant term give
$\int g\,d\nu=1$.  Since $g^+=(|g|+g)/2$, integration yields
\[
 \int g^+\,d\nu=\frac12\left(\int|g|\,d\nu+\int g\,d\nu\right)
 =\frac{M(g)+1}{2}>\frac{M(g)}2.
\]
Finally, feasibility gives $g(t)\leq0$ on
$\operatorname{supp}\nu\cap[-1,s]$, so $g^+$ vanishes there.
\end{proof}

In the binary case, $t>s$ means Hamming distance less than $d$.
The corresponding Krawtchouk mass-concentration statement is
\cite[Theorem~3.4]{Salmon2026HalfRate}; its application to positive
mass in the Delsarte program appears in the proof of
\cite[Proposition~3.8]{Salmon2026HalfRate}.

\begin{theorem}[Finite-dimensional Delsarte lower bound]
\label{thm:finite-barrier}
Let $R\geq0$ and $-1<s<1$.  If, for some $\sigma>0$,
\begin{equation}\label{eq:finite-barrier-criterion}
 q_{n,R}(\sigma)+\frac{\ln2}{n}\leq s\sigma,
\end{equation}
then
\[
 \frac1n\ln D_X(s)\geq R.
\]
In particular, in the binary and spherical settings respectively,
\[
 \frac1n\log_2 D_2(n,d)\geq\frac R{\ln2}
 \quad\text{when }s=1-2d/n,
 \qquad \frac1n\ln D_{\mathbb S}(n,s)\geq R.
\]
\end{theorem}
\begin{proof}
Suppose that $g$ is feasible at $s$ and $R(g)\leq R$.
The finite mass estimate \eqref{eq:finite-mass-concentration} and
\eqref{eq:finite-barrier-criterion} give
\[
 \int_{t>s}|g|\,d\nu
 \leq M(g)e^{n(q_{n,R}(\sigma)-s\sigma)}
 \leq\frac{M(g)}2.
\]
This contradicts Lemma~\ref{lem:positive-mass}, since $g^+$ is
supported on $t>s$.  Thus every feasible $g$ has $R(g)>R$.
As $M(g)\geq|\int g\,d\nu|=1$, we have
$g(1)=M(g)e^{nR(g)}>e^{nR}$.  Taking the infimum proves the claim.
\end{proof}

\subsection{The asymptotic lower bound}

Put
\[
 \Theta=\bigl(\limsup_{n\to\infty}\Xi_n\bigr)^*.
\]
Since the subharmonic functions $\Xi_n$ are locally uniformly bounded
above by $\mathsf h$,
\cite[Theorem~2.4.7(2)]{DinhSibony2005} gives that
$\Theta$ is subharmonic and $\Theta\leq\mathsf h$; also $\Theta(0)=0$
since $\Xi_n(0)=\mathsf h(0)=0$.
The limiting function has an explicit form in both settings.

\begin{example}[Binary codes]
Since \eqref{eq:binary-Xi} is independent of $n$ and continuous,
\begin{equation}\label{eq:binary-Theta}
 \Theta(\lambda)=\ln\max\{|\cosh\lambda|,|\sinh\lambda|\}.
\end{equation}
In particular, $\Theta(\mathrm i\eta)=
\ln\max\{|\cos\eta|,|\sin\eta|\}$.
\end{example}

Computing the limit in the spherical case is a bit more involved.

\begin{lemma}[Locally uniform spherical asymptotics]
\label{lem:spherical-Xi-convergence}
In the spherical setting, $\Xi_n(\lambda)$ converges locally uniformly on
$\{\operatorname{Re}\lambda>0\}$ to
\[
 \frac12\left(
 \operatorname{Re}\sqrt{1+4\lambda^2}-1
 -\ln\left|\frac{1+\sqrt{1+4\lambda^2}}2\right|\right),
\]
where the square root has positive real part.
\end{lemma}
\begin{proof}
To justify the supremum over degrees, fix a compact set
$\mathcal K\subset\{\operatorname{Re}\lambda>0\}$ and write
$a_{n,j}=(j+\alpha_n)/n$.  For $a>0$ define
\[
 Z(a,\lambda)=\sqrt{a^2+\lambda^2}
 +a\log\frac{\lambda}{a+\sqrt{a^2+\lambda^2}},\qquad
 F(a,\lambda)=-\frac12(1+\ln|\lambda|)+\operatorname{Re}Z(a,\lambda),
\]
with branches continued from positive real $\lambda$.
The uniform Debye expansion~\cite[Eq.~10.41.3 and \S10.41(iii)]{DLMF}
gives
\begin{equation}\label{eq:uniform-bessel-debye}
 I_{na}(n\lambda)=
 \frac{e^{nZ(a,\lambda)}}{\sqrt{2\pi n}\,(a^2+\lambda^2)^{1/4}}
 \bigl(1+O_{\mathcal K}(n^{-1})\bigr),\qquad a=a_{n,j},
\end{equation}
uniformly for $\lambda\in\mathcal K$ and all $j\geq0$ as $n\to\infty$.
Indeed, the scaled argument $\lambda/a$ remains in a fixed sector
$|\arg z|\leq\pi/2-\varepsilon$, where that expansion is uniform,
and the order $na=j+\alpha_n$ is at least $\alpha_n\sim n/2$.
This includes $j/n\to\infty$, when $\lambda/a\to0$.

The prefactor $(a^2+\lambda^2)^{-1/4}$ is uniformly bounded for
$a\geq\alpha_n/n$ and $\lambda\in\mathcal K$.  Combining
\eqref{eq:uniform-bessel-debye} with Stirling's formula therefore gives
\[
 |\beta_j(\lambda)|\leq C_{\mathcal K}e^{nF(a_{n,j},\lambda)},\qquad
 \frac1n\ln|\beta_0(\lambda)|
   =F(a_{n,0},\lambda)+O_{\mathcal K}(n^{-1}),
\]
for all sufficiently large $n$, with $C_{\mathcal K}$ independent of $j$.
Moreover,
\[
 \frac{\partial F}{\partial a}
 =-\operatorname{Re}\log
       \frac{a+\sqrt{a^2+\lambda^2}}{\lambda}<0.
\]
The strict sign follows because the logarithm is the principal
$\operatorname{arcsinh}(a/\lambda)$, whose real part is positive in this
half-plane.  The upper bound for every $j$ and the lower bound from $j=0$
now give
\[
 \Xi_n(\lambda)=F(a_{n,0},\lambda)+O_{\mathcal K}(n^{-1})
 \longrightarrow F(\tfrac12,\lambda)
\]
locally uniformly in the right half-plane.  Evaluating $F(\tfrac12,\lambda)$
gives the stated formula.

\end{proof}

\begin{example}[Spheres]\label{ex:spherical-Theta}
For $\operatorname{Re}\lambda>0$, put
$w(\lambda)=\sqrt{1+4\lambda^2}$, with the branch having positive
real part.  Then
\begin{equation}\label{eq:spherical-Theta}
 \Theta(\lambda)=\frac12\left(
   \operatorname{Re}w(\lambda)-1
   -\ln\left|\frac{1+w(\lambda)}2\right|\right).
\end{equation}
On the left half-plane use $\Theta(-\lambda)=\Theta(\lambda)$;
on the imaginary axis use the continuous boundary values.  Explicitly,
\begin{equation}\label{eq:spherical-Theta-axis}
 \Theta(\mathrm i\eta)=
 \begin{cases}
 \displaystyle\frac12\left(\sqrt{1-4\eta^2}-1
       -\ln\frac{1+\sqrt{1-4\eta^2}}2\right),
       & |\eta|\leq\frac12,\\[4pt]
 \displaystyle-\frac12(1+\ln|\eta|),&|\eta|\geq\frac12.
 \end{cases}
\end{equation}
In particular, $\Theta(0)=0$.

Lemma~\ref{lem:spherical-Xi-convergence} gives the formula in the right
half-plane, and symmetry gives the left half-plane.  Let $\widetilde\Theta$ denote the
continuous extension of the displayed formula to the imaginary axis,
with values \eqref{eq:spherical-Theta-axis}.  We have $\Theta=\widetilde\Theta$ off
that axis.  At a point on the axis, upper semicontinuity gives
$\Theta\geq\widetilde\Theta$, while the submean inequality on shrinking disks gives
$\Theta\leq\widetilde\Theta$: their area integrals agree since the axis has zero area.
Thus equality holds everywhere.
\end{example}

\begin{lemma}[Spherical decay]\label{lem:spherical-perron-tails}\label{lem:spherical-tails}
Fix $R\geq0$ and $\sigma>0$.  There is a constant $C_{R,\sigma}$,
independent of $n$ and $\eta$, such that
\begin{equation}\label{eq:spherical-pointwise-tail}
 V_{n,R}(\sigma+\mathrm i\eta)
 \leq C_{R,\sigma}-\frac14\ln|\eta|
 \qquad(n\geq4,\ |\eta|\geq1).
\end{equation}
\end{lemma}
\begin{proof}
Schl\"afli's integral~\cite[Eq.~10.9.6]{DLMF} gives, for real $v\geq0$
and $\operatorname{Re}z>0$,
\[
 J_v(z)=\frac1\pi\int_0^\pi\cos(z\sin t-vt)\,dt
       -\frac{\sin\pi v}{\pi}\int_0^\infty e^{-z\sinh t-vt}\,dt.
\]
The first term has modulus at most $e^{|\operatorname{Im}z|}$;
using $\sinh t\geq t$ bounds the second by $1/(\pi\operatorname{Re}z)$.
Apply the connection formula between $I_v$ and $J_v$
\cite[Eq.~10.27.6]{DLMF} with $z=n|\eta|-\mathrm i n\sigma$, using
conjugation when $\eta<0$.  Uniformly in $v\geq0$ and $|\eta|\geq1$,
\[
 |I_v(n(\sigma+\mathrm i\eta))|
 \leq e^{n\sigma}+\frac1{\pi n|\eta|}\leq2e^{n\sigma}.
\]
Applying Stirling's formula to the Gamma factor in
\eqref{eq:spherical-Xi} gives
\[
 \frac1n\ln\left[\Gamma(\alpha_n+1)
       \left(\frac2n\right)^{\alpha_n}\right]
 =-\frac12+O\left(\frac{\ln n}{n}\right).
\]
Since $\alpha_n/n=(n-2)/(2n)\geq1/4$ for $n\geq4$ and
$|\sigma+\mathrm i\eta|\geq|\eta|$, the bound
$V_{n,R}\leq R+\Xi_n$ gives the assertion.
\end{proof}

For $R\geq0$, define the obstacle
\[
 \psi_R(\lambda)=\min\{\mathsf h(\lambda),R+\Theta(\lambda)\}.
\]
Consider the obstacle problem of finding its largest subharmonic minorant:
\begin{equation}\label{eq:perron-problem}
 \mathcal V(R)=\{v\text{ subharmonic on }\mathbb C:
       v\leq\psi_R\},
 \qquad
 V_R=\bigl(\sup_{v\in\mathcal V(R)}v\bigr)^*.
\end{equation}
This construction, due to Perron, is classical in potential theory; see
\cite[Chapters~2 and~4]{Ransford1995}.  Here the obstacle comes from the two
bounds on the Laplace transform.
Since $\mathsf h$ is continuous, $V_R\leq\mathsf h$.
Define the slope
\[
 m_R(\sigma)=\sup_{\eta\in\mathbb R}V_R(\sigma+\mathrm i\eta),
 \qquad c(R)=\inf_{\sigma>0}\frac{m_R(\sigma)}\sigma.
\]

\begin{theorem}[Delsarte rate lower bound]\label{thm:barrier}
Let $R\geq0$ and $-1<s<1$.  If $c(R)<s$, then
\[
 \liminf_{n\to\infty}\frac1n\ln D_X(s)\geq R,
\]
where $X=H_n$ in the binary case and $X=S^{n-1}$ in the spherical case.
In the notation of the introduction,
\[
 \underline R_D\bigl(\tfrac{1-s}{2}\bigr)\geq\frac R{\ln2}
 \qquad\text{and}\qquad
 \underline R_D^{\mathbb S}(s)\geq R,
\]
respectively.
\end{theorem}

\begin{lemma}
\label{lem:fourier-factor-asymptotics}
For fixed $R\geq0$ and $\sigma>0$, in either setting,
\[
 \limsup_{n\to\infty}q_{n,R}(\sigma)\leq m_R(\sigma).
\]
\end{lemma}
\begin{proof}
For any compact interval $I$ of positive length, bounding the integrand
by its supremum gives
\[
 \frac1{2n}\ln\int_I e^{2nV_{n,R}(\sigma+\mathrm i\eta)}\,d\eta
 \leq\sup_{\eta\in I}V_{n,R}(\sigma+\mathrm i\eta)
       +\frac{\ln|I|}{2n}.
\]
The prefactor $L_nc_n$ in \eqref{eq:finite-fourier-factor} contributes
only $O(\ln n/n)$.  For binary codes, take $I=[0,\pi]$ and use
$V_{n,R}=V_R$ to obtain the claim.

For spheres, fix any real $a>m_R(\sigma)$.
Let $C=C_{R,\sigma}$ be the constant in
Lemma~\ref{lem:spherical-perron-tails}, and choose $A\geq1$ so that
$C-(\ln A)/4<a$.  Integrating the pointwise decay bound gives
\begin{equation}\label{eq:spherical-tail-majorant}
 \int_{|\eta|>A}e^{2nV_{n,R}(\sigma+\mathrm i\eta)}\,d\eta
 \leq\frac{2e^{2nC}}{n/2-1}\,A^{1-n/2}
 \qquad(n\geq4).
\end{equation}
Taking logarithms and dividing by $2n$ shows that, for all sufficiently
large $n$,
\begin{equation}\label{eq:spherical-perron-tail}
 \int_{|\eta|>A}e^{2nV_{n,R}(\sigma+\mathrm i\eta)}\,d\eta
 \leq e^{2na}.
\end{equation}
The subharmonic upper limit
\[
 \overline V=(\limsup_n V_{n,R})^*\leq\min\{\mathsf h,R+\Theta\}
\]
satisfies $\overline V\leq V_R$ by maximality of the Perron solution.
By upper semicontinuity, $V_R<a$ on a neighborhood of the compact segment
$\sigma+\mathrm i[-A,A]$.  Hartogs' lemma~\cite[Theorems~2.4.7--2.4.9]{DinhSibony2005},
applied there, gives
$\sup_{|\eta|\leq A}V_{n,R}(\sigma+\mathrm i\eta)\leq a$
for all sufficiently large $n$ (also if $\overline V\equiv-\infty$).
Splitting the integral at $A$
bounds it by $(2A+1)e^{2na}$.  Hence
$\limsup_n q_{n,R}(\sigma)\leq a$, and letting $a\downarrow m_R(\sigma)$
proves the assertion.
\end{proof}

\begin{proof}[Proof of Theorem~\ref{thm:barrier}]
Since $c(R)<s$, choose $\sigma>0$ with $m_R(\sigma)<s\sigma$.
Lemma~\ref{lem:fourier-factor-asymptotics} implies that, for all
sufficiently large $n$,
\[
 q_{n,R}(\sigma)+\frac{\ln2}{n}\leq s\sigma.
\]
The strict gap absorbs the asymptotic error, $\ln2/n$, and the
$O(1/n)$ rounding of the binary threshold.
Theorem~\ref{thm:finite-barrier} therefore gives
$\frac1n\ln D_X(s)\geq R$ for all sufficiently large $n$.
Taking the lower limit proves the claim.
\end{proof}

\begin{remark}[Dropping subharmonic gives Gilbert--Varshamov]
\label{rem:gv-obstacle}
If the subharmonic hypothesis in \eqref{eq:perron-problem} is dropped,
the envelope is simply $\psi_R$.  Since $\mathsf h(\sigma+\mathrm i\eta)=\sigma$
for $\sigma>0$ and $\Theta(\sigma+\mathrm i\eta)\le\Theta(\sigma)$,
its slope is
\[
 c_{\mathrm{rel}}(R)=\inf_{\sigma>0}
 \min\left\{1,\frac{R+\Theta(\sigma)}\sigma\right\}.
\]
The inequality for $\Theta$ follows from the absolute-value bound on
its Laplace coefficients; equality holds on the real axis.
For $0<s<1$, therefore,
\[
 c_{\mathrm{rel}}(R)<s
 \quad\Longleftrightarrow\quad
 R<\sup_{\sigma>0}\{s\sigma-\Theta(\sigma)\}.
\]
For binary codes $\Theta(\sigma)=\ln\cosh\sigma$, and this Legendre
transform is
$\ln2-H_{\mathrm{bin}}((1-s)/2)$, where
$H_{\mathrm{bin}}(p)=-p\ln p-(1-p)\ln(1-p)$.
Thus this relaxation gives the Gilbert--Varshamov exponent
$1-H_{\mathrm{bin}}(\delta)/\ln2$ in bits, where $\delta=(1-s)/2$.
For spheres,
\[
 \Theta(\sigma)=\sup_{-1<t<1}
 \left\{\sigma t+\tfrac12\ln(1-t^2)\right\},
\]
so the transform is $-\tfrac12\ln(1-s^2)$, the spherical
cap-volume (Shannon) exponent in nats.
Because $V_R\le\psi_R$, these relaxed slopes already give lower bounds
through Theorem~\ref{thm:barrier}, by approaching the strict rate threshold
from below.
\end{remark}

\section{The elliptic solution of the obstacle problem}\label{sec:elliptic-obstacle}

The previous section established an obstacle problem, phrased as maximization over
subharmonic functions less than or equal to $\psi_R=\min\{\mathsf h,R+\Theta\}$.  In this section, we solve this obstacle problem and give
an expression in terms of elliptic integrals.  To do this,
we construct a real potential $\mathcal W$,
normalized by $\mathcal W(0)=0$, and prove that it equals $V_R$.

\subsection{The solution to the obstacle problem}

The obstacle problem from the previous section lives in the ``$\lambda$-plane'',
the space underlying Laplace transformations of feasible functions for the Delsarte
program.  By conformal invariance, we will conformally map a strip in the $\lambda$-plane to a ``$w$ half-plane'',
the right half-plane, where the binary and spherical code obstacle problems can be solved in a
uniform way.

\begin{definition}[Coordinates and obstacles]\label{def:elliptic-coordinates}
Use $w=\lambda$ for spheres and $w=\sinh(2\lambda)$ on the binary
half-strip $\Re\lambda>0$, $|\Im\lambda|<\pi/4$.  Write $W$ for the
potential in the $w$-coordinate, so that
$\mathcal W(\lambda)=W(\lambda)$ for spheres and
$\mathcal W(\lambda)=W(\sinh(2\lambda))$ for binary codes before
reflection.  The binary obstacles in this coordinate are
\[
 H(w)=\tfrac12\Re\operatorname{arcsinh}w,\qquad
 T(w)=\tfrac12\ln\left|\frac{1+\sqrt{1+w^2}}2\right|;
\]
these are $\mathsf h$ and $\Theta$, respectively, after the change of
variables.  Branches are continued from the positive real axis.
For spheres $\Theta(\mathrm iy)=-\tfrac12-\tfrac12\ln y$ when
$y\ge1/2$; for binary codes $T(\mathrm iy)=\tfrac12\ln(y/2)$ when
$y\ge1$.
\end{definition}

The contact set must be supported on the imaginary axis.  Consider the case of spherical codes first.  Since $\mathsf h$ and $\Theta$ are both harmonic on the right half-plane, $-\psi_R$ is subharmonic there.  By the strong maximum principle, either $V_R < \psi_R$ on the whole half-plane, or they coincide.  Since they cannot coincide, we cannot have any contact on the right half-plane, similarly the left half-plane, and the contact set must be supported on the imaginary axis.  The same logic works for binary codes in the $w$-coordinate as well.  This suggests the following proposal for contact set.

\begin{definition}[Contact set]\label{def:elliptic-contacts}
Choose $0<\kappa<1$, with $K,E,K',E'$ as in the introduction, and put
$\ell=\pi K'/K$.  For $0<a<b$ with $a/b=\kappa$, the proposed contact set
in the $w$-plane is
\begin{equation}\label{eq:spherical-contact-guess}
 C_{a,b}=\mathrm i[-a,a]\cup\{\mathrm iy:|y|\ge b\}.
\end{equation}
We will choose $a$ and $b$ as follows:
\[
\begin{array}{c|cc|c}
 \text{branch}&a&b&\text{free boundary endpoints}\\ \hline
 \text{spherical}&\kappa K/\pi&K/\pi&\mathrm ia,\ \mathrm ib\\
 \text{binary }+&1&1/\kappa&\mathrm ib\\
 \text{binary }-&\kappa&1&\mathrm ia
\end{array}
\]
The binary endpoint at $w=\mathrm i$ is fixed by the corner $\lambda=\mathrm i\pi/4$ of the half-strip.
\end{definition}

For ease of reference, we will refer to the segments $\mathrm i[a,b]$ and $\mathrm i[-b,-a]$ as the gaps.

\begin{definition}[Elliptic potential]
\label{def:elliptic-potential}
In the right half-plane define
\[
 P(w)=\sqrt{w^2+a^2}\sqrt{w^2+b^2},\qquad
 \rho(t)=\frac1{\sqrt{(t^2-a^2)(t^2-b^2)}}\quad(t>b),
\]
where both square roots in $P$ are positive on the positive real axis.
The substitution $t=b/\sin\theta$ gives
\begin{equation}\label{eq:elliptic-density-mass}
 \int_b^\infty\rho(t)\,dt
 =\frac1b\int_0^{\pi/2}\frac{d\theta}{\sqrt{1-\kappa^2\sin^2\theta}}
 =\frac Kb.
\end{equation}
Set
\[
 I(w)=\int_b^\infty\frac{\rho(t)}{t^2+w^2}\,dt.
\]
We now define the harmonic primitive $W$ as follows:
\begin{equation}\label{eq:elliptic-primitive}
 W(w)=\Re\int_0^w\Phi(z)\,dz,\qquad
 \Phi_{\mathrm S}(w)=\frac{P(w)I(w)}\pi,\qquad
 \Phi_{\mathrm B}(w)=\frac{P(w)}\pi
       \left(\frac{K}{b(1+w^2)}-I(w)\right).
\end{equation}
$W$ is normalized by $W(0)=0$.  It is path independent in the
right half-plane, and continuous on its finite boundary.  Reflect evenly
across the imaginary axis in the spherical case.  In the binary case
pull back by $w=\sinh(2\lambda)$, then reflect across the axis and the
horizontal sides of the half-strip.  Denote the resulting functions of
$\lambda$ by $\mathcal W_{\mathrm S}$ and $\mathcal W_{\mathrm B}$.
The binary extension is $\mathrm i\pi/2$-periodic.
\end{definition}

Consider the involution $\iota(w)=ab/w$ of $\mathbb{P}^1$
that preserves $C_{a,b}\cup\{\infty\}$, exchanging the central
contact segment with the outer rays and the upper and lower gaps with
their endpoints reversed.  This involution lifts to the smooth
compactification of $Y^2=(w^2+a^2)(w^2+b^2)$, which we write as $E_{a,b}$.
\[
 \iota_E:E_{a,b}\longrightarrow E_{a,b},\qquad
 (w,Y)\longmapsto\left(\frac{ab}{w},\frac{ab}{w^2}Y\right).
\]
With origin $O=(0,ab)$, the same map is subtraction in the elliptic
group law: $\iota_E(p)=T-p$, where $T$ is the point over infinity
with $Y/w^2\to1$ and $2T=O$.

On the sheet $Y=P(w)$, with boundary values from the right half-plane,
write $\alpha=\Phi(w)\,dw$.
Let $\gamma$ be the lift of the lower gap from $(-\mathrm ia,0)$
to $(-\mathrm ib,0)$ on the sheet $Y=P(w)$, using boundary values
from the right half-plane.
Since $W(-\mathrm ia)=0$ and $W(-\mathrm ib)=W(\mathrm ib)$,
\[
 W(\mathrm ib)=\Re\int_\gamma\alpha.
\]

The central segment has $W=0$ because $\Phi$ is real there, and
$W(\overline w)=W(w)$ by conjugation symmetry.

The differential $\alpha$ combines rational functions and an integral $I$
related to elliptic integrals of the third kind.  However, the symmetry under
$\iota_E$ will allow us to cancel the contribution of $I$.

\begin{proposition}
\label{prop:elliptic-gap-inversion}
For each of the spherical and binary branches,
\[
 -W(\mathrm ib)=\Re\int_\gamma\iota_E^*\alpha.
\]
\end{proposition}
\begin{proof}
The image $\iota_E\circ\gamma$ is the lift of the upper gap on the
same sheet, oriented from $(\mathrm ib,0)$ to $(\mathrm ia,0)$.
Thus the change of variables formula gives
\[
 \Re\int_\gamma\iota_E^*\alpha
 =\Re\int_{\iota_E\circ\gamma}\alpha
 =W(\mathrm ia)-W(\mathrm ib)=-W(\mathrm ib).
\]
\end{proof}

The above proposition reduces the computation to the integral of
$\alpha-\iota_E^*\alpha$.  Lemma~\ref{lem:elliptic-differential-inversion}
gives a simple rational expression for this differential in both
spherical and binary cases.

\begin{lemma}\label{lem:elliptic-differential-inversion}
On the chosen sheet, and on the gaps by boundary limits,
\[
\begin{aligned}
 \alpha_{\mathrm S}-\iota_E^*\alpha_{\mathrm S}
 &=\left(\frac{Y}{w^2}-\frac1{2w}\right)dw,\\
 \alpha_{\mathrm B\pm}-\iota_E^*\alpha_{\mathrm B\pm}
 &=\pm\frac{Kb\kappa'^2}{\pi}\frac{dw}{Y}
   +\frac12\frac{dw}{w}.
\end{aligned}
\]
In particular, these differences extend to meromorphic differentials
on $E_{a,b}$.
\end{lemma}
\begin{proof}
The integral representation of $I$ gives the identity
\[
 I(w)+\frac{a^2b^2}{w^4}I\!\left(\frac{ab}{w}\right)
 =\frac{K}{bw^2}-\frac{\pi}{2wP(w)}.
\]
Indeed, for real $w>0$, combining the integrands and differentiating
the arctangent gives
\begin{align*}
 &I(w)+\frac{a^2b^2}{w^4}I\!\left(\frac{ab}{w}\right)
       -\frac{K}{bw^2}\\
 &\quad=\int_b^\infty
 \frac{(a^2b^2-t^4)\rho(t)}
      {(t^2+w^2)(w^2t^2+a^2b^2)}\,dt\\
 &\quad=\frac1{wP(w)}\int_b^\infty\frac{d}{dt}
       \arctan\!\left(\frac{tP(w)\rho(t)}w\right)dt
 =-\frac{\pi}{2wP(w)}.
\end{align*}
The identity extends to the right half-plane by analytic continuation.
Since $\iota_E^*dw=-ab\,dw/w^2$, pulling back
$\alpha_{\mathrm S}=YI(w)\,dw/\pi$ yields
\[
 \alpha_{\mathrm S}-\iota_E^*\alpha_{\mathrm S}
 =\left(\frac{KY}{\pi bw^2}-\frac1{2w}\right)dw.
\]
The spherical scale $b=K/\pi$ gives the stated formula.

For binary codes, $\alpha_{\mathrm B}=KY\,dw/(\pi b(1+w^2))
-\alpha_{\mathrm S}$, so the same calculation gives
\[
 \alpha_{\mathrm B}-\iota_E^*\alpha_{\mathrm B}
 =\frac{KY}{\pi b}
   \left(\frac1{1+w^2}-\frac1{w^2+a^2b^2}\right)dw
   +\frac12\frac{dw}{w}.
\]
Using $Y^2=(w^2+a^2)(w^2+b^2)$, this simplifies to the
stated formulas when $a=1$ for the plus branch and $b=1$ for the
minus branch, since $b^2-a^2=b^2\kappa'^2$.
\end{proof}

\begin{proposition}[Boundary data]\label{prop:elliptic-boundary}
The potential $W$ vanishes on the central segment.  It has zero normal
derivative on the free gaps, and its values on the rays are
\begin{equation}\label{eq:elliptic-ray-data}
 W(\mathrm iy)=
 \begin{cases}
 W(\mathrm ib)-\tfrac12\ln(y/b),&\text{spherical},\\
 W(\mathrm ib)+\tfrac12\ln(y/b),&\text{binary},
 \end{cases}\qquad y\ge b.
\end{equation}
\end{proposition}
\begin{proof}
On the central segment $\Phi$ is real, so the tangential derivative
of $W$ is zero; normalization gives $W=0$.  On the gaps $\Phi$ is
purely imaginary, giving the zero normal derivative needed for harmonic
reflection.

For $w=\mathrm iy$ with $y>b$, inversion sends $w$ to
$-\mathrm iab/y$ on the central segment, so $\Phi(ab/w)$ is real.
Moreover, $Y=P(\mathrm iy+0)$ is real, and
\[
 \iota_E^*\alpha=-\frac{ab}{w^2}\Phi(ab/w)\,dw.
\]
The pullback has a real coefficient on the ray, as do $Y\,dw/w^2$
and $dw/Y$.  Taking imaginary parts in
Lemma~\ref{lem:elliptic-differential-inversion} therefore leaves only
the terms proportional to $dw/w$, giving
\[
 \Im\Phi_{\mathrm S}(\mathrm iy+0)=\frac1{2y},\qquad
 \Im\Phi_{\mathrm B}(\mathrm iy+0)=-\frac1{2y}.
\]
Since $dW(\mathrm iy)/dy=\Re(\mathrm i\Phi)=-\Im\Phi$,
integration from $b$ gives \eqref{eq:elliptic-ray-data}.
\end{proof}

\begin{lemma}[Growth at infinity]\label{lem:elliptic-growth}
Uniformly in the closed right half-plane, as $|w|\to\infty$,
\begin{equation}\label{eq:elliptic-derivative-asymptotics}
 \Phi_{\mathrm S}(w)=1-\frac1{2w}+O(w^{-2}),\qquad
 \Phi_{\mathrm B}(w)=\frac1{2w}+O(w^{-2}),
\end{equation}
and hence
\begin{equation}\label{eq:elliptic-growth}
 W_{\mathrm S}(w)=\Re w-\tfrac12\ln|w|+O(1),\qquad
 W_{\mathrm B}(w)=\tfrac12\ln|w|+O(1).
\end{equation}
\end{lemma}
\begin{proof}
Lemma~\ref{lem:elliptic-differential-inversion} gives the functional equation
\[
 \Phi_{\mathrm S}(w)+\frac{ab}{w^2}\Phi_{\mathrm S}\Bigl(\frac{ab}w\Bigr)
 =\frac{P(w)}{w^2}-\frac1{2w},\qquad
 \Phi_{\mathrm B\pm}(w)+\frac{ab}{w^2}\Phi_{\mathrm B\pm}\Bigl(\frac{ab}w\Bigr)
 =\frac1{2w}\pm\frac{Kb\kappa'^2}{\pi P(w)}.
\]
The disk $|z|<a$ contains no branch point of $P$, no singularity of
$I$, and no pole of $(1+z^2)^{-1}$, so $\Phi$ is holomorphic there.
By the functional equation, the terms $P(w)/w^2-1$ and $1/P(w)$ are holomorphic on
$b<|w|\le\infty$ and vanish at infinity to second order.
Considering boundary values on the imaginary axis gives \eqref{eq:elliptic-derivative-asymptotics}.

Fix $R>b$, and write $w=re^{\mathrm i\theta}$, where $r\ge R$ and
$|\theta|\le\pi/2$.  Integrate the derivative asymptotics along the
radial segment from $Re^{\mathrm i\theta}$ to $w$.  The leading terms
give $\Re w-\frac12\ln|w|$ in the spherical case and
$\frac12\ln|w|$ in the binary case.  The error is uniformly bounded because
\[
 \int_R^r O(t^{-2})\,dt=O(1).
\]
Since $W$ is bounded on the fixed semicircle $|w|=R$, the growth
estimates in~\eqref{eq:elliptic-growth} are uniform in the closed
right half-plane.  Finally,
$\ln|\sinh2\lambda|=2\Re\lambda-\ln2+O(e^{-4\Re\lambda})$, so
$\mathcal W_{\mathrm B}(\lambda)=\Re\lambda+O(1)$ on the binary
half-strip.
\end{proof}

\begin{theorem}[Elliptic solution: rate and threshold]\label{thm:elliptic-obstacle}
For each $0<\kappa<1$, the construction above equals the Perron envelope
$V_R$ of \eqref{eq:perron-problem}, at the following rates.  In the binary
case put
\[
 R_\pm(\kappa)=\frac12\ln2\pm B(\kappa).
\]
Then
\begin{equation}\label{eq:binary-elliptic-threshold}
 c(R_+(\kappa))=\frac{2E}{\pi},\qquad
 c(R_-(\kappa))=
       \frac{2\bigl(E-\kappa'^2K\bigr)}{\pi\kappa}.
\end{equation}
At the rate $R=\frac12\ln2$, one has $c(R)=2/\pi$.
Equivalently, $R_\pm=\ln2\,\mathcal J(\delta_\pm)$ and
$c(R_\pm)=1-2\delta_\pm$.  In the spherical case put
\[
 R_{\mathrm S}(\kappa)=
 \frac12+\frac12\ln\frac K\pi+\frac14\ln\kappa
 +\frac{K\bigl((1+\kappa^2)K'-2E'\bigr)}{2\pi}.
\]
Then
\begin{equation}\label{eq:spherical-elliptic-threshold}
 c(R_{\mathrm S}(\kappa))=\frac{K-E}{\kappa K}.
\end{equation}
Equivalently, $R_{\mathrm S}(\kappa)=F_{\mathrm S}(s(\kappa))$ and
$c(R_{\mathrm S}(\kappa))=s(\kappa)$.
\end{theorem}

The proof of Theorem~\ref{thm:elliptic-obstacle} occupies the next two
subsections.  The half-rate statement follows from the convergence
of the $R_+$ and $R_-$ branches.  The next subsection will prove that
$\mathcal W=V_R$, while the final subsection will compute the explicit
rates and thresholds required above.

\subsection{Edwards' duality and identification of the solution}

The convex duality statement for this obstacle problem is Edwards' duality between
subharmonic minorants and Jensen measures.  We identify the candidate by
the extended maximum principle.

\begin{theorem}[Extended maximum principle {\cite[Theorem~3.6.9]{Ransford1995}}]
\label{thm:extended-maximum-principle}
Let $D\subseteq\mathbb C$ be a domain with nonpolar boundary, and let
$u$ be subharmonic on $D$ and bounded above.  If
\[
 \limsup_{\substack{z\to\zeta\\z\in D}}u(z)\le0
\]
for every $\zeta\in\partial D$ outside a polar set, then $u\le0$ on $D$.
Here the boundary is taken in the Riemann sphere; in particular, infinity
may be included in the exceptional set.
\end{theorem}

For $\lambda_0\in\mathbb C$, let $J_{\lambda_0}$ be the set of compactly
supported probability measures $\mu$ satisfying
\[
 u(\lambda_0)\le\int u\,d\mu
 \qquad\text{for every subharmonic }u\text{ on }\mathbb C.
\]
These are the Jensen measures at $\lambda_0$.  Edwards' duality
\cite{Edwards1966}, in its noncompact subharmonic form
\cite[\S1]{GogusPerkinsPoletsky2013}, gives
\begin{equation}\label{eq:edwards-obstacle-duality}
 V_R(\lambda_0)
 =\sup_{v\in\mathcal V(R)}v(\lambda_0)
 =\inf_{\mu\in J_{\lambda_0}}\int\psi_R\,d\mu.
\end{equation}
The obstacle is continuous in both settings, so this form of the theorem
applies.  The inequality from left to right follows directly from
$v(\lambda_0)\le\int v\,d\mu\le\int\psi_R\,d\mu$.
We first establish that $\mathcal W$ is an admissible minorant.
Throughout the following two lemmas, let $R$ satisfy
\begin{equation}\label{eq:elliptic-endpoint-matching}
 W(\mathrm ib)=
 \begin{cases}
 R+\Theta(\mathrm ib),&\text{spherical},\\
 R+T(\mathrm ib),&\text{binary}.
 \end{cases}
\end{equation}
Since $b\ge1/2$ on the sphere and $b\ge1$ in the binary case,
Proposition~\ref{prop:elliptic-boundary} then gives equality with the
corresponding obstacle on the rays.

\begin{lemma}[Admissibility of the elliptic potential]\label{lem:elliptic-obstacle-check}
The functions $\mathcal W$ are globally subharmonic and satisfy
$\mathcal W\le\min\{\mathsf h,R+\Theta\}$.  Equality with the
obstacle holds on the contact set.
\end{lemma}
\begin{proof}
\emph{Subharmonicity.} The construction is harmonic off the imaginary
axis in the $w$-plane.  On the gaps, $\Re\Phi=0$ by
Proposition~\ref{prop:elliptic-boundary}, so even reflection is harmonic
across them.  Thus $\Delta\mathcal W$ is supported on the pullbacks
and reflections of $C_{a,b}$, and its positivity as a measure reduces to showing
$\Re\Phi\ge0$ on $C_{a,b}$, which we check by cases.
On $\mathrm i(-a,a)$, the formula in
\eqref{eq:elliptic-primitive} represents $\Phi$ as a positive integral in the spherical code case.
For binary codes, \eqref{eq:elliptic-density-mass} gives
\[
 \Phi_{\mathrm B}(w)=\frac{P(w)}{\pi(1+w^2)}
 \int_b^\infty\frac{(t^2-1)\rho(t)}{t^2+w^2}\,dt.
\]
The prefactor and integrand are positive at $w=\mathrm iy$ with
$|y|<a\le1$, proving this positivity in both cases.
For the outer rays, consider a point $\mathrm iy$ and set
$c=\sqrt{y^2-b^2}$ and $t=\sqrt{b^2+u^2}$.  Its principal-value integral
has the form
\[
 \operatorname{pv}\int_0^\infty\frac{f(u)}{u^2-c^2}\,du
 =\int_0^\infty\frac{f(u)-f(c)}{u^2-c^2}\,du.
\]
For spheres $f=1/(t\sqrt{t^2-a^2})$ is decreasing.  For binary codes
$f=(t^2-1)/(t\sqrt{t^2-a^2})$ is increasing, but the additional factor
$(1-y^2)^{-1}$ is negative.  Since $P(\mathrm iy)<0$, both cases give
$\Re\Phi\ge0$.  Even reflection has a nonnegative derivative jump.

Summing up, away from the branch points, the reflected function is subharmonic.
By the removable-singularity theorem for subharmonic functions, the subharmonic property extends across these points as well by continuity.
The result transfers to $\mathcal{W}$ in the binary code case by pulling back along $w=\sinh(2\lambda)$.

\emph{Obstacle inequalities.} We check $W \le H$ and $W \le R + T$ on the imaginary axis, then extend to the interior by the maximum principle.
We note $W = H$ on the contact segment and $W = R + T$ on the rays.
We check both inequalities on the gaps, after which the complementary inequalities on the contact segment and rays follow by monotonicity.  Differentiating along $w=\mathrm iy$ gives the comparisons
\begin{equation}\label{eq:elliptic-gap-comparisons}
\begin{array}{ll}
 \text{spherical:}&
 0\le-\dfrac{d}{dy}W(\mathrm iy)
 \le\dfrac{1-\sqrt{1-y^2/b^2}}{2y},\\[6pt]
 \text{binary }-:&
 0\le-\dfrac{d}{dy}W(\mathrm iy)
 \le\dfrac{(1-y^2)^{-1/2}-1}{2y},\\[6pt]
 \text{binary }+:&
 \dfrac1{2y}\le\dfrac{d}{dy}W(\mathrm iy)
 \le\dfrac1{2\sqrt{y^2-1}}.
\end{array}
\end{equation}
For the first two comparisons use
$\sqrt{y^2-a^2}/\sqrt{t^2-a^2}\le y/t$ for $t\ge b>y$.
For the last use
$t/y\le\sqrt{t^2-1}/\sqrt{y^2-1}$ for the lower bound and
$\sqrt{t^2-1}\le t$ for the upper bound.  The remaining integrals reduce,
by $t^2=b^2+u^2$, to integrals of the form $\int_0^\infty(u^2+c^2)^{-1}du=\pi/(2c)$.

On the sphere $\mathsf h(\mathrm iy)=0$ and
\[
 -\frac{d}{dy}\Theta(\mathrm iy)=
 \begin{cases}(1-\sqrt{1-4y^2})/(2y),&0<y<1/2,\\
 1/(2y),&y>1/2.
 \end{cases}
\]
Because $b\ge1/2$, the first comparison shows that
$R+\Theta-W$ decreases across the gap to its value zero at $b$.
Also $W$ decreases from zero at $a$.  Thus both obstacle inequalities
hold on the gap.  They hold on the central segment because $\Theta$
is decreasing, and on the rays because $W$ continues decreasing.

On the binary minus gap, $H=0$ and
$-T_y=((1-y^2)^{-1/2}-1)/(2y)$.  The same argument applies.
Here $W(\mathrm i)<0$, so $R<\tfrac12\ln2$; on the rays
$H=\tfrac12\operatorname{arccosh}y\ge R+\tfrac12\ln(y/2)$.
On the plus gap the last comparison is $T_y\le W_y\le H_y$.
Since $H-W=0$ at $a=1$ and $R+T-W=0$ at $b$, both differences
are nonnegative.  Integrating the lower bound gives
$R\ge\tfrac12\ln2$, which also checks the central segment.
On the rays $H_y\ge T_y$, so the support inequality persists.

Finally both obstacles are harmonic in the right half-plane in the
spherical case, and in the binary fundamental half-strip.  Their
differences from $W$ have at most logarithmic growth by
\eqref{eq:elliptic-growth}.  The half-plane maximum principle for
sublinear growth extends the boundary inequalities to the interior
(one may subtract $\epsilon\Re w^p$, $0<p<1$, and let $\epsilon\downarrow0$).
Reflection gives the obstacle inequalities globally.  Together with the values on the contact set, this proves the lemma.
\end{proof}

\begin{lemma}[Identification by the maximum principle]\label{lem:elliptic-maximality}
For each branch, $V_R=\mathcal W$ when
\eqref{eq:elliptic-endpoint-matching} holds.
\end{lemma}
\begin{proof}
By Lemma~\ref{lem:elliptic-obstacle-check}, $\mathcal W$ is an admissible
minorant and is harmonic off its contact set.
We let $C$ be the preimage of the contact set in the $\lambda$-plane.
The growth estimate \eqref{eq:elliptic-growth} bounds
$\psi_R-\mathcal W\le R+\Theta-\mathcal W$ above near infinity.
For binary codes the bound is uniform in the imaginary direction by
periodicity of $\Theta$ and $\mathcal W$.
For any subharmonic minorant $v\le\psi_R$, the difference
$v-\mathcal W$ is therefore subharmonic and bounded above away from $C$,
with nonpositive upper boundary limits at every finite boundary point.
The set $C$ contains a nondegenerate segment, so each component of its
complement has nonpolar boundary.  Theorem~\ref{thm:extended-maximum-principle}, with infinity as a polar
exception, gives $v\le\mathcal W$ off $C$; this inequality already holds
on $C$.  Thus $V_R=\mathcal W$.
\end{proof}

\subsection{Evaluation of rates and thresholds}

We can therefore relate $W$ directly to
the rates in Theorem~\ref{thm:elliptic-obstacle}.

\begin{lemma}[Evaluation of the rate]\label{lem:elliptic-rate-evaluation}
For the binary branches, with $b=1/\kappa$ and $b=1$, respectively,
\[
 R_\pm(\kappa)=\tfrac12\ln2\pm B(\kappa)
 =W_{\mathrm B\pm}(\mathrm ib)+\tfrac12\ln(2/b).
\]
For the spherical branch, with $b=K/\pi$,
\begin{align*}
 R_{\mathrm S}(\kappa)
 &=\frac12+\frac12\ln\frac K\pi+\frac14\ln\kappa
   +\frac{K\bigl((1+\kappa^2)K'-2E'\bigr)}{2\pi}\\
 &=W_{\mathrm S}(\mathrm ib)+\tfrac12+\tfrac12\ln b.
\end{align*}
Thus the rates in Theorem~\ref{thm:elliptic-obstacle} satisfy
\eqref{eq:elliptic-endpoint-matching}.
\end{lemma}
\begin{proof}
Proposition~\ref{prop:elliptic-gap-inversion} gives
\[
 W(\mathrm ib)=\frac12\Re\int_\gamma
                   (\alpha-\iota_E^*\alpha).
\]
Along $\gamma$, write $w=-\mathrm iy$ and
$Y=-\mathrm i\sqrt{(y^2-a^2)(b^2-y^2)}$, with $a\le y\le b$.
The substitution $y^2=b^2-(b^2-a^2)\sin^2\theta$ gives
\[
 \int_\gamma\frac{dw}{Y}=\frac{K'}b,\qquad
 \int_\gamma\frac{w^2\,dw}{Y}=-bE',\qquad
 \int_\gamma\frac{dw}{w}=\ln\frac ba=-\ln\kappa.
\]
The curve equation also gives
\[
 \frac{Y}{w^2}\,dw
 =(2w^2+a^2+b^2)\frac{dw}{Y}-d\!\left(\frac Yw\right).
\]
Since $Y=0$ at both endpoints of $\gamma$,
\[
 \int_\gamma\frac{Y}{w^2}\,dw
 =b\bigl((1+\kappa^2)K'-2E'\bigr).
\]
Applying Lemma~\ref{lem:elliptic-differential-inversion} therefore yields
\begin{align*}
 W_{\mathrm S}(\mathrm ib)
 &=\frac{K\bigl((1+\kappa^2)K'-2E'\bigr)}{2\pi}
   +\frac14\ln\kappa,\\
 W_{\mathrm B\pm}(\mathrm ib)
 &=\pm\frac{\kappa'^2KK'}{2\pi}-\frac14\ln\kappa,
\end{align*}
where $b=K/\pi$ in the spherical case.
Finally, $\Theta(\mathrm ib)=-\tfrac12-\tfrac12\ln b$ and
$T(\mathrm ib)=\tfrac12\ln(b/2)$.
Subtracting these obstacle values and using the two binary scales
$b=1/\kappa$ and $b=1$ gives the stated rates.
\end{proof}

\begin{lemma}\label{lem:elliptic-slope}
For every $\sigma>0$,
\[
 \sup_{\eta\in\mathbb R}V_R(\sigma+\mathrm i\eta)=\mathcal W(\sigma).
\]
The function $\sigma\mapsto\mathcal W(\sigma)$ is convex on $[0,\infty)$.
Consequently $c(R)=\mathcal W'(0+)$.
\end{lemma}
\begin{proof}
To show the maximum statement, we will show that derivative $\partial_\eta\mathcal W(\sigma+\mathrm i\eta) \le 0$ for $\eta \ge 0$ in both the binary and spherical case.  This derivative is automatically harmonic
in the upper quadrant for spheres, and in the upper quarter-strip for
binary codes and equal to zero on the real axis.  Its values on the imaginary
axis are nonpositive by \eqref{eq:elliptic-gap-comparisons}.  On the upper boundary $\Im\lambda=\pi/4$ of the binary half-strip, we compute the derivative
\[
 \partial_\eta\mathcal W(\sigma+\mathrm i\pi/4)
 =-2\sinh(2\sigma)\Re\Phi_{\mathrm B}
                  (\mathrm i\cosh(2\sigma)+0)\le0.
\]
For the spherical case, the derivative tends to zero at infinity by
Lemma~\ref{lem:elliptic-growth}.  For binary codes,
$w=\sinh(2\lambda)$ gives the complex derivative
$2\cosh(2\lambda)\Phi_{\mathrm B}(w)
=\coth(2\lambda)+O(e^{-2\Re\lambda})$,
uniformly in the closed quarter-strip, so its imaginary part tends to zero.
Since the obstacle in both cases is $1$-Lipschitz, the solution to the obstacle problem is as well (the proof of~\cite[Theorem~3.8]{Silvestre2007} carries over word-for-word).  Then $|\partial_\eta \mathcal W|\le 1$, which supplies the needed regularity to apply the maximum principle.  We conclude that $\partial_\eta\mathcal W\le0$.
Evenness and binary periodicity consequently give
\[
 \sup_{\eta\in\mathbb R}V_R(\sigma+\mathrm i\eta)=\mathcal W(\sigma).
\]

For each $\sigma>0$, the function
$\eta\mapsto\mathcal W(\sigma+\mathrm i\eta)$ has a maximum at $\eta=0$,
so $\partial_\eta^2\mathcal W(\sigma,0)\le0$.
Since $\mathcal W$ is harmonic near the positive real axis,
\[
 \frac{d^2}{d\sigma^2}\mathcal W(\sigma)
 =-\partial_\eta^2\mathcal W(\sigma,0)\ge0.
\]
Thus $\mathcal W$ is convex on $[0,\infty)$ by continuity at zero.
Its secant slopes $\mathcal W(\sigma)/\sigma$ are therefore increasing,
and $\mathcal W(0)=0$ gives
\[
 c(R)=\inf_{\sigma>0}\frac{\mathcal W(\sigma)}\sigma
     =\mathcal W'(0+).
\]
\end{proof}

\begin{proof}[Completion of Theorem~\ref{thm:elliptic-obstacle}]
The substitution $t=b/u$ gives
$\int_b^\infty t^{-2}\rho(t)dt=(K-E)/(\kappa^2b^3)$.
Hence the derivatives at zero are $abI(0)/\pi$ on the sphere and
$2ab(K/b-I(0))/\pi$ in the binary case.  By
Lemma~\ref{lem:elliptic-slope}, these derivatives equal $c(R)$.
Substitution of the three scales proves the thresholds in
Theorem~\ref{thm:elliptic-obstacle}.

All displayed quantities are continuous on $0<\kappa<1$.
The plus threshold decreases strictly because $E$ does.  The minus
threshold increases strictly, as seen from
\[
 \frac{E-(1-\kappa^2)K}{\kappa}
 =\kappa\int_0^{\pi/2}
       \frac{\cos^2\theta}{\sqrt{1-\kappa^2\sin^2\theta}}\,d\theta.
\]
The spherical threshold is $\kappa$ times the mean of $\sin^2\theta$
under density proportional to $(1-\kappa^2\sin^2\theta)^{-1/2}$;
increasing $\kappa$ strictly increases this mean.  Since $\psi_R$ and hence $V_R$ increase with $R$, $c(R)$ is
nondecreasing, so the rates have the same strict ordering as their
thresholds.
\end{proof}

\section{Asymptotic upper bounds}\label{sec:upper-bounds}

The approach to the upper bound extends the constructions in~\cite[Chapter~2]{OpenAI2026} for spherical codes and~\cite[Section~4.2]{Salmon2026HalfRate} for binary codes.  While these constructions only establish a sphere packing and half-rate limit, respectively, in loc.~cit., we will do a more detailed analysis of both constructions.  In particular, we first draw a parallel between the two as living on bosonic, resp. fermionic Fock space and involving a pair $(A, \gamma)$ of matrix and distinguished direction.  This then gives a natural formulation in terms of interlacing spectra which leads to a variational problem.  Finally, we solve this variational problem and arrive at an upper bound that matches the lower bound from the previous section.

\subsection{Upper bounds from interlacing spectra}

Let $A$ be a real symmetric $N\times N$ matrix, let
$\gamma\in\mathbb R^N$ be a unit vector, and choose an isometry
$U:\mathbb R^{N-1}\to\gamma^\perp$.  Set $B=U^{\mathsf T}AU$.
Write $x_1\ge\cdots\ge x_N$ and $y_1\ge\cdots\ge y_{N-1}$ for the
spectra of $A$ and $B$, respectively.  Cauchy interlacing gives
\[
 x_1\ge y_1\ge x_2\ge\cdots\ge y_{N-1}\ge x_N.
\]
Conversely, every interlacing pair $(x; y)$ is realized by a pair $(A, \gamma)$.

Define the binary and Bose entropy functions by
\[
 H_{\mathrm{bin}}(p)=-p\ln p-(1-p)\ln(1-p),\qquad
 H_{\mathrm{Bose}}(p)=(p+1)\ln(p+1)-p\ln p,
\]
with $0\ln0=0$, and put
\[
\begin{array}{c|c|c|c}
 \star&I_\star&f_\star(x)&\varphi_\star(x)\\ \hline
 \mathrm B&[0,1]&H_{\mathrm{bin}}\bigl((1-\sqrt x)/2\bigr)&\sqrt{1-x}\\[2pt]
 \mathrm S&[0,\infty)&H_{\mathrm{Bose}}\bigl((\sqrt{1+4x}-1)/2\bigr)
   &\sqrt{4x/(1+4x)}
\end{array}
\]
For a datum with spectrum in $I_\star$, its cost and correlation are
\begin{equation}\label{eq:upper-matrix-statistics}
 C_\star(A,\gamma)=\operatorname{tr}f_\star(A)-\operatorname{tr}f_\star(B),
 \qquad
 S_\star(A,\gamma)=\gamma^{\mathsf T}\varphi_\star(A)\gamma.
\end{equation}
Thus the cost is $\sum_j f_\star(x_j)-\sum_k f_\star(y_k)$.
If $A=\operatorname{diag}(x_1,\ldots,x_N)$, the correlation is
$\sum_j\gamma_j^2\varphi_\star(x_j)$.

The spherical construction follows the organization of the $N$-qubit
construction in \cite[Section~4.2]{Salmon2026HalfRate}: choose a Gaussian
state and a linear observable, then reduce their statistics to the
one-particle matrix $A$.  The bosonic operation retaining the compressed
matrix $B$ is removal of one oscillator.

\begin{definition}[$N$-mode bosonic construction]\label{def:bosonic-construction}
Let $A\succeq0$.  On bosonic Fock space
$\mathcal F_N=\bigoplus_{m\ge0}\operatorname{Sym}^m\mathbb C^N$, let
$c_j,c_j^\dagger$ be annihilation and creation operators, and put
\[
 q_j=(c_j+c_j^\dagger)/\sqrt2,\qquad
 p_j=(c_j-c_j^\dagger)/(\mathrm i\sqrt2).
\]
They satisfy $[q_j,p_k]=\mathrm i\delta_{jk}I$ and
$[q_j,q_k]=[p_j,p_k]=0$.
Let $|0\rangle$ be the vacuum vector and $\mathcal N_{0,2A}$ be the
centered Gaussian probability measure with covariance $2A$.
Define the density operator
\[
 \rho_A=\int_{\mathbb R^N}
 e^{\mathrm i k^{\mathsf T}q}|0\rangle\langle0|
 e^{-\mathrm i k^{\mathsf T}q}\,d\mathcal N_{0,2A}(k),
\]
which is a mixture of coherent states with respect to the Gaussian density.
Each coherent state $e^{\mathrm i k^{\mathsf T}q}|0\rangle$ has mean
quadratures $(0,k)$ and covariance $I_{2N}/2$.  Since $\mathbb E k=0$,
the mixture is centered, so its covariance is the expectation of the
symmetrized quadrature products.  Averaging the second moments gives
\[
 \mathcal V_A=\tfrac12 I_{2N}
 +\mathbb E_k\begin{pmatrix}0&0\\0&kk^{\mathsf T}\end{pmatrix}
 =\tfrac12\begin{pmatrix}I_N&0\\0&I_N+4A\end{pmatrix},
\]
using $\mathbb E(kk^{\mathsf T})=2A$.

Select $q_\gamma=\sum_j\gamma_jq_j$ and let
$O=\begin{pmatrix}\gamma^{\mathsf T}\\U^{\mathsf T}\end{pmatrix}$,
so that $O\gamma=e_1$.  Write
\[
 \mathcal F_N=\mathcal F_1\otimes\mathcal F_{N-1}.
\]
Rotating the $\gamma$-oscillator to the first factor and tracing over
that factor retains the coherent-state displacement $U^{\mathsf T}k$,
whose Gaussian covariance is $2B$.  Thus
\[
 \operatorname{tr}_{\mathcal F_1}\rho_{OAO^{\mathsf T}}=\rho_B.
\]
For a density operator write
$\operatorname{Ent}(\rho)=-\operatorname{tr}(\rho\ln\rho)$.
Define its entropy cost and correlation by
\[
 \operatorname{Cost}=\operatorname{Ent}(\rho_A)-\operatorname{Ent}(\rho_B),
 \qquad
 \operatorname{Corr}=
 \frac{\operatorname{tr}(\rho_A^{1/2}q_\gamma\rho_A^{1/2}q_\gamma)}
      {\operatorname{tr}(\rho_Aq_\gamma^2)}.
\]
For $N=1$, $\mathcal F_0=\mathbb C$ and $\operatorname{Ent}(\rho_B)=0$.
\end{definition}

\begin{proposition}\label{prop:adjacent-degree-weights}
On one-mode Fock space, let $|m\rangle$ be the normalized vector of
degree $m$ and put
\[
 \tau_r=(1-r)\sum_{m\ge0}r^m|m\rangle\langle m|,
 \qquad 0<r<1.
\]
Let $Q$ be a self-adjoint observable satisfying
\[
 \langle m|Q|n\rangle=0\qquad\text{unless }|m-n|=1,
 \qquad m,n\ge0.
\]
If $0<\operatorname{tr}(\tau_rQ^2)<\infty$, then
\[
 \frac{\operatorname{tr}(\tau_r^{1/2}Q\tau_r^{1/2}Q)}
      {\operatorname{tr}(\tau_rQ^2)}
 =\frac{2\sqrt r}{1+r}.
\]
In particular, this holds for $Q=q$.
\end{proposition}
\begin{proof}
Write $\lambda_m=(1-r)r^m$ and
$Q_{m,m+1}=\langle m|Q|m+1\rangle$.  Pairing the two directions of
each transition gives
\begin{align*}
 \operatorname{tr}(\tau_r^{1/2}Q\tau_r^{1/2}Q)
 &=\sum_{m\ge0}2\sqrt{\lambda_m\lambda_{m+1}}\,|Q_{m,m+1}|^2,\\
 \operatorname{tr}(\tau_rQ^2)
 &=\sum_{m\ge0}(\lambda_m+\lambda_{m+1})|Q_{m,m+1}|^2.
\end{align*}
Since $\lambda_{m+1}=r\lambda_m$, the ratio of the two weights is
$2\sqrt r/(1+r)$ on every transition.  The matrix coefficients cancel.
For $q$, finiteness follows from $|q_{m,m+1}|^2=(m+1)/2$.
\end{proof}

\begin{lemma}[Bosonic entropy and correlation]\label{lem:bosonic-statistics}
The construction of Definition~\ref{def:bosonic-construction} satisfies
\[
 \operatorname{Cost}=C_{\mathrm S}(A,\gamma),\qquad
 \operatorname{Corr}=S_{\mathrm S}(A,\gamma).
\]
In particular, the cost is the conditional entropy of the first factor
given the second in $\rho_{OAO^{\mathsf T}}$.
\end{lemma}
\begin{proof}
The symplectic eigenvalues of $\mathcal V_A$ are
$\sqrt{1+4x_j}/2$, so its normal-mode occupations are
$a_j=(\sqrt{1+4x_j}-1)/2$.  The Gaussian entropy formula
\cite[Section~II.C]{Weedbrook2012} gives
$\operatorname{Ent}(\rho_A)=\sum_jH_{\mathrm{Bose}}(a_j)$.
The same argument for $B$ gives
$\operatorname{Ent}(\rho_B)=\operatorname{tr}f_{\mathrm S}(B)$,
proving the cost identity.  Unitary rotation preserves entropy, giving
the conditional entropy interpretation.

For the correlation, diagonalize $A$.  Each mode is conjugate by a
unitary dilation to $\tau_{r_j}$, where
\[
 \nu_j=\frac{\sqrt{1+4x_j}}2,\qquad
 r_j=\frac{2\nu_j-1}{2\nu_j+1}.
\]
Dilation rescales $q$, multiplying both traces by the same factor.
Since $q$ raises or lowers degree by one,
Proposition~\ref{prop:adjacent-degree-weights} gives, for $x_j>0$, their ratio as
$2\sqrt{r_j}/(1+r_j)=\sqrt{1-(2\nu_j)^{-2}}$.
In our normalization every $q_j$ has variance $1/2$, and distinct
modes have zero cross terms in both expressions.  Thus
\[
 \operatorname{Corr}=\sum_j\gamma_j^2
 \sqrt{\frac{4x_j}{1+4x_j}}
 =\gamma^{\mathsf T}\varphi_{\mathrm S}(A)\gamma.
\]
Vacuum modes contribute zero to the numerator, so the argument also
covers zero eigenvalues.
\end{proof}

\begin{theorem}[$N$-mode spherical bound]\label{thm:spherical-matrix-upper}
For the bosonic construction above and $-1<s<1$, if
$s<\operatorname{Corr}$, then
\[
 \overline R_D^{\mathbb S}(s)\le\operatorname{Cost}
 =C_{\mathrm S}(A,\gamma).
\]
\end{theorem}
\begin{proof}
The moving-subspace construction of~\cite[Chapter~2, Theorem~6.2]{OpenAI2026} bounds the Delsarte linear program.
In the language there, let $J_n$ be the symmetric transition matrix of the construction, and let
$L_n$ be the dimension of the ambient space divided by the rank of the
projection.  The resulting Delsarte bound is
\begin{equation}\label{eq:spherical-moving-subspace-bound}
 D_{\mathbb S}(n,s)\le
 \frac{(1-s)L_n}{\lambda_{\max}(J_n)-s},
 \qquad \lambda_{\max}(J_n)>\max\{s,0\}.
\end{equation}
Thus it suffices to realize a limiting transition eigenvalue at least
$\operatorname{Corr}$ and a rate $\lim_n n^{-1} \ln L_n$ equal to
$\operatorname{Cost}$.

Initially assume $A\succ0$ and strictly interlacing spectra.
Put $a_i=(\sqrt{1+4x_i}-1)/2$ and
$b_j=(\sqrt{1+4y_j}-1)/2$.
The highest-weight construction in~\cite[Chapter~2, Lemmas~7.1--7.2,
Equations~(87)--(88), and Lemma~A.1]{OpenAI2026} supplies $J_n,L_n$ with
\begin{align*}
 \liminf_n\lambda_{\max}(J_n)
 &\ge\sum_i w_i\sqrt{\frac{4x_i}{1+4x_i}},&
 w_i&=\frac{\prod_j(x_i-y_j)}{\prod_{k\ne i}(x_i-x_k)},\\
 \lim_n\frac1n\ln L_n
 &=\sum_iH_{\mathrm{Bose}}(a_i)-\sum_jH_{\mathrm{Bose}}(b_j)
 =\operatorname{Cost}.
\end{align*}
The cofactor identity
\[
 \gamma^{\mathsf T}(zI-A)^{-1}\gamma
 =\frac{\det(zI-B)}{\det(zI-A)}
\]
identifies the residues $w_i$ with $\|P_{x_i}\gamma\|^2$.
By Lemma~\ref{lem:bosonic-statistics}, the transition eigenvalue therefore
has lower limit at least $\operatorname{Corr}$.
Since $A\succ0$, we have $\operatorname{Corr}>0$, and hence
\[
 \liminf_n\lambda_{\max}(J_n)
 \ge\operatorname{Corr}>\max\{s,0\}.
\]
Thus~\eqref{eq:spherical-moving-subspace-bound} applies for all sufficiently
large $n$, with its denominator bounded away from zero.
Taking logarithms and dividing by $n$ leaves only the dimension-ratio
exponent $\operatorname{Cost}$.

The case of general $A\succeq0$ and arbitrary $\gamma$ follows by a density argument: approximate $A$ by positive definite matrices with
simple spectrum and $\gamma$ by unit vectors having nonzero component
in every eigenspace.
\end{proof}

The construction for binary codes uses fermionic Fock space via the $N$-qubit construction, introduced in the case of half rate in~\cite[Section~4.2]{Salmon2026HalfRate}.  We give the general construction below.

We recall Alrabiah and Guruswami's pretty good criterion~\cite[Theorem~1]{AG2026}, which gives upper bounds for binary codes.
The Delsarte version is established
in~\cite[Section~4.1, Equation~(4.18), and Theorem~2.20]{Salmon2026HalfRate}.

\begin{proposition}[Pretty good criterion for the Delsarte bound]\label{prop:pgm-delsarte}
Let $\sigma_+,\sigma_-$ be density matrices on a finite-dimensional Hilbert
space, regarded as equiprobable channel outputs.  Put
\[
 \varrho=\tfrac12(\sigma_++\sigma_-),\qquad
 M_\pm=\tfrac12\varrho^{-1/2}\sigma_\pm\varrho^{-1/2},
\]
where inverses are taken on $\operatorname{supp}\varrho$.
The Holevo information in bits and the error probability of the pretty
good measurement $(M_+,M_-)$ are
\[
\begin{aligned}
 \chi&=\frac{\operatorname{Ent}(\varrho)
 -\tfrac12\operatorname{Ent}(\sigma_+)
 -\tfrac12\operatorname{Ent}(\sigma_-)}{\ln2},\\
 p_{\mathrm e}&=\tfrac12\operatorname{tr}(\sigma_+M_-)
              +\tfrac12\operatorname{tr}(\sigma_-M_+).
\end{aligned}
\]
For $0<\delta<\frac12$, if $p_{\mathrm e}<\delta$, then
\[
 \overline R_D(\delta)\le\chi.
\]
\end{proposition}

\begin{theorem}[$N$-qubit projection bound]\label{thm:horn-upper}
For $0\preceq A\preceq I$ and $0<\delta<\frac12$, if
$1-2\delta<S_{\mathrm B}(A,\gamma)$, then
\begin{equation}\label{eq:horn-rate-upper}
 \overline R_D(\delta)\le\frac{C_{\mathrm B}(A,\gamma)}{\ln2}.
\end{equation}
\end{theorem}
\begin{proof}
By Proposition~\ref{prop:pgm-delsarte}, it suffices to construct equiprobable
output states with Holevo information and decoding error
$\chi = C_{\mathrm B}(A,\gamma) / \ln 2$ and $p_{\mathrm e} = \frac{1 - S_{\mathrm B}(A,\gamma)}{2}$.

Diagonalize $A=\operatorname{diag}(x_1,\ldots,x_N)$ and initially assume
$x_j<1$.  Set $\eta_j=\sqrt{x_j}$ and $D=\operatorname{diag}(\eta_j)$.
Define the state, involution,
and projections of the $N$-qubit construction as
\[
 \varrho=2^{-N}\prod_j(I-\mathrm i\eta_ja_jb_j),\qquad
 J_\gamma=\sum_j\gamma_ja_j,\qquad
 \Pi_\pm=(I\pm J_\gamma)/2.
\]
Here $a_j,b_j$ are self-adjoint Clifford generators on
$(\mathbb C^2)^{\otimes N}$, with
$\{a_j,a_k\}=\{b_j,b_k\}=2\delta_{jk}I$ and $\{a_j,b_k\}=0$.
Define $\sigma_\pm=2\varrho^{1/2}\Pi_\pm\varrho^{1/2}$.
Parity exchanges $\sigma_+$ and $\sigma_-$, so
$\sigma_\pm$ are output-symmetric states with average $\varrho$.
Their pretty good measurement is $(\Pi_+,\Pi_-)$.
\cite[Theorem~4.5, Equation~(4.36)]{Salmon2026HalfRate}
gives
\[
 1-2p_{\mathrm e}
 =\operatorname{tr}(J_\gamma\varrho^{1/2}J_\gamma\varrho^{1/2})
 =\gamma^{\mathsf T}(I-D^2)^{1/2}\gamma
 =S_{\mathrm B}(A,\gamma).
\]

It remains to identify the Holevo information.  The product state has
$\operatorname{Ent}(\varrho)=\sum_jf_{\mathrm B}(x_j)$.
Put $L=(I-\gamma\gamma^{\mathsf T})D$ and
$\varrho_e=(\varrho+J_\gamma\varrho J_\gamma)/2$.
The identity for the part invariant under conjugation by
$J_\gamma$~\cite[Equation~(4.40)]{Salmon2026HalfRate}
says that $\varrho_e$ is the fermionic Gaussian state with coupling matrix $L$.
The singular values of $L$ are $\sqrt{y_1},\ldots,\sqrt{y_{N-1}}, 0$, counting multiplicities,
since $LL^{\mathsf T}$ restricts to $B$ on $\gamma^\perp$.
The Gaussian entropy formula therefore gives
\[
 \operatorname{Ent}(\varrho_e)
 =f_{\mathrm B}(0)+\sum_{j=1}^{N-1}f_{\mathrm B}(y_j)
 =\ln2+\sum_{j=1}^{N-1}f_{\mathrm B}(y_j).
\]
On the other hand,
\[
 \varrho_e=\Pi_+\varrho\Pi_++\Pi_-\varrho\Pi_-.
\]
These two orthogonal blocks each have trace $1/2$, and
$2\Pi_\pm\varrho\Pi_\pm$ have the same nonzero spectra as
$\sigma_\pm$.  Hence
\[
 \operatorname{Ent}(\varrho_e)
 =\ln2+\tfrac12\operatorname{Ent}(\sigma_+)
       +\tfrac12\operatorname{Ent}(\sigma_-).
\]
It follows that
\[
 \chi=\frac{\operatorname{Ent}(\varrho)
 -\tfrac12\operatorname{Ent}(\sigma_+)
 -\tfrac12\operatorname{Ent}(\sigma_-)}{\ln2}
 =\frac{C_{\mathrm B}(A,\gamma)}{\ln2}.
\]
The bound now follows from the pretty good criterion.
\end{proof}

There are several clear parallels between the above constructions.  They begin with a moving-subspace construction, pass to Gaussian quantum information theory, and end at interlacing spectra.  In the following subsection, we will explore a variational problem over measures that arises from interlacing spectra via the Markov--Krein correspondence.

\subsection{The Markov--Krein correspondence and the measure problem}

We now introduce ideas of Kerov~\cite[Section~2]{Kerov2000} to pass from
interlacing spectra to probability measures.

For a Borel subset $I\subseteq\mathbb R$, write $\mathcal P(I)$ for the
space of Borel probability measures on $I$.  The spectral measure of $A$ at $\gamma$ is
$\mathfrak m=\sum_jw_j\delta_{x_j}$, where
$w_j=\|P_{x_j}\gamma\|^2$ and $P_{x_j}$ is the spectral projection.
The sum here runs over distinct eigenvalues.  This is a probability
measure.  Its Cauchy transform satisfies
\begin{equation}\label{eq:upper-markov-krein}
 G_{\mathfrak m}(z)=\int\frac{d\mathfrak m(x)}{z-x}
 =\gamma^{\mathsf T}(zI-A)^{-1}\gamma
 =\frac{\prod_k(z-y_k)}{\prod_j(z-x_j)}.
\end{equation}
The products count eigenvalues with multiplicity.  To see the last equality,
choose an orthonormal basis whose first vector is $\gamma$; the corresponding
diagonal cofactor of $zI-A$ is $\det(zI-B)$.

After canceling common eigenvalues, the remaining spectra strictly interlace.
Partial fractions give the weights explicitly:
\[
 w_j=\frac{\prod_k(x_j-y_k)}{\prod_{i\ne j}(x_j-x_i)}.
\]
Interlacing makes these weights positive, and
$G_{\mathfrak m}(z)\sim1/z$ gives $\sum_jw_j=1$.
Conversely, the Cauchy transform of a finite atomic probability measure
has exactly one zero between each pair of consecutive atoms: it is
strictly decreasing there, from $+\infty$ to $-\infty$.
Its poles and zeros therefore recover the interlacing pair.
The choice $A=\operatorname{diag}(x_j)$ and $\gamma_j=\sqrt{w_j}$
realizes this pair by matrices.

This bijection between strictly interlacing pairs and finite atomic
probability measures is the discrete Markov--Krein correspondence.  The measure is called
the transition measure of the interlacing pair; in our matrix realization,
it is exactly the spectral measure of $A$ at $\gamma$.

With this correspondence in mind, we recast the optimization problem in terms of measures.  Define
\[
 g_{\mathfrak m}(u)=\int\frac{d\mathfrak m(x)}{1-u^2x},\qquad
 T_{\mathfrak m}(r)=\int\frac{\operatorname{sech}^2r}
 {\operatorname{sech}^2r+4x}\,d\mathfrak m(x).
\]
Here $0\le u<1$, $r\ge0$, and $\operatorname{sech}r=1/\cosh r$.
The measure costs and correlations are
\begin{align}
 C_{\mathrm B}(\mathfrak m)&=\ln2-\frac12\int_0^1
                  \frac{\ln g_{\mathfrak m}(u)}{u^2}\,du,
 &S_{\mathrm B}(\mathfrak m)&=\int\sqrt{1-x}\,d\mathfrak m(x),
 \label{eq:binary-measure-statistics}\\
 C_{\mathrm S}(\mathfrak m)&=-\frac12\int_0^\infty
             \operatorname{sech}^2r\ln T_{\mathfrak m}(r)\,dr,
 &S_{\mathrm S}(\mathfrak m)&=\int\sqrt{\frac{4x}{1+4x}}\,d\mathfrak m(x).
 \label{eq:spherical-measure-statistics}
\end{align}
These agree with \eqref{eq:upper-matrix-statistics} for finite atomic
measures.  To see the cost identities, use the scalar entropy
representations
\[
 f_{\mathrm B}(x)=\ln2+\frac12\int_0^1
       \frac{\ln(1-u^2x)}{u^2}\,du,\qquad
 f_{\mathrm S}(x)=\frac12\int_0^\infty
       \operatorname{sech}^2r\ln(1+4x\cosh^2r)\,dr.
\]
Each identity follows by differentiation in $x$ and its value at $x=0$;
continuity gives the endpoints.  The determinant identity
\eqref{eq:upper-markov-krein} gives
\[
 g_{\mathfrak m}(u)=\frac{\det(I-u^2B)}{\det(I-u^2A)},\qquad
 T_{\mathfrak m}(r)=
 \frac{\det(I+4\cosh^2r B)}{\det(I+4\cosh^2r A)}.
\]
Subtracting the scalar entropy integrals proves the cost identities.

\begin{definition}[Optimization over measures]
\label{def:binary-measure-variational}\label{def:spherical-measure-variational}
For $\star\in\{\mathrm B,\mathrm S\}$ and $0<s<1$, put
\begin{equation}\label{eq:upper-unified-measure-variational}
 \mathcal U_\star(s)=\inf\{C_\star(\mathfrak m):
 \mathfrak m\in\mathcal P(I_\star),
 \ S_\star(\mathfrak m)\ge s\}.
\end{equation}
\end{definition}
Both are convex problems: correlation is linear in the measure, and cost
is a convex functional.

\begin{lemma}[Weak continuity]\label{lem:upper-continuity}
Both $S_\star$ and $C_\star$ are continuous for weak (narrow) convergence on
$\mathcal P(I_\star)$.
\end{lemma}
\begin{proof}
The correlation integrands are
bounded and continuous.  In the binary case, the integral
converges pointwise, and for $0<u<1$,
\[
 0\le\ln g_{\mathfrak m}(u)\le-\ln(1-u^2),
\]
whose upper bound divided by $u^2$ is integrable.  For spheres, if
$\mathfrak m_j\to\mathfrak m$ weakly, choose $M<\infty$ and $c>0$
such that $\mathfrak m_j([0,M])\ge c$ for all sufficiently large $j$.
Then $T_{\mathfrak m_j}(r)\to T_{\mathfrak m}(r)$ pointwise and
\[
 0\le-\ln T_{\mathfrak m_j}(r)
 \le-\ln c+\ln(1+4M\cosh^2r),
\]
which is integrable against $\operatorname{sech}^2r\,dr$.  Dominated
convergence proves cost continuity in both settings.
\end{proof}

Finite atomic probability measures are weakly dense in
$\mathcal P(I_\star)$, so by Lemma~\ref{lem:upper-continuity} every
measure admits finite atomic approximants whose costs and correlations
converge to its own.

\begin{lemma}[Measure upper bound]\label{lem:upper-approximation}
For $0<s<1$,
\[
 \ln2\,\overline R_D((1-s)/2)\le\mathcal U_{\mathrm B}(s),\qquad
 \overline R_D^{\mathbb S}(s)\le\mathcal U_{\mathrm S}(s).
\]
\end{lemma}
\begin{proof}
Suppose $S_\star(\mathfrak m)\ge s$.
Mix $\mathfrak m$ with $\varepsilon\delta_0$ in the binary case,
or with $\varepsilon\delta_L$ for a finite $L$ satisfying
$\varphi_{\mathrm S}(L)>s$ in the spherical case.  The resulting correlation
is strictly greater than $s$.  Convexity bounds its cost by
$(1-\varepsilon)C_\star(\mathfrak m)+\varepsilon f_\star(L)$
(with $L=0$ in the binary case).
By Lemma~\ref{lem:upper-continuity}, approximate this mixture by finite atomic
measures, preserving the strict inequality.  Apply the matrix bound
and then let $\varepsilon\downarrow0$.
Taking the infimum proves the assertion.
\end{proof}

\subsection{Optimization over measures and elliptic integrals}

We are now ready to solve the optimization problem over measures from the previous subsection and make our second contact with elliptic integrals.  Recall $K,K',E,E'$ from the introduction and the spacing
$\ell=\pi K'/K$ from Definition~\ref{def:elliptic-contacts}.  Put
\begin{equation}\label{eq:upper-elliptic-spacing}
 d_\ell=\frac{2K}\pi.
\end{equation}
To motivate the binary measure, one could write out KKT first-order optimality conditions and trust their global optimality since the underlying problem over measures is convex.  We will not do so fully and instead give a related ansatz.  Our measure must be supported on $I = [0,1]$, and we will assume that $\mathfrak{m}$ is a discrete measure whose support has a single accumulation point at $1$.  A priori, the function $g_{\mathfrak m}$ is holomorphic on the slit domain $\Omega = \mathbb{C} \setminus ((-\infty,-1] \cup [1,\infty))$, and we make the following substitution
\[
  H_{\mathfrak m}(r)=g_{\mathfrak m}(\tanh r)/\cosh r.
\]
We seek the simplest meromorphic map compatible with this function.

Under the assumptions above, $H_{\mathfrak m}$ has a meromorphic continuation and satisfies
\[
 H_{\mathfrak m}(-r)=H_{\mathfrak m}(r),\qquad
 H_{\mathfrak m}(r+\mathrm i\pi)=-H_{\mathfrak m}(r),\qquad
 H_{\mathfrak m}(0)=1.
\]
Writing $x=\tanh^2\xi$ with $\xi\ge0$, an atom at $x<1$
contributes to $H_{\mathfrak m}$ through the kernel
\begin{equation}\label{eq:upper-binary-kernel}
 \frac1{\cosh r(1-\tanh^2r\tanh^2\xi)}
 =\frac{\cosh\xi}{2}
   \bigl(\operatorname{sech}(r-\xi)+\operatorname{sech}(r+\xi)\bigr).
\end{equation}
Thus the poles of $H_{\mathfrak m}$ lie on $\Im r=\pi/2$
modulo $\mathrm i\pi$.
Supply the real period $\ell$, and seek a meromorphic map
\[
 H:E_\ell\longrightarrow\mathbb P^1,\qquad
 E_\ell=\mathbb C/(\ell\mathbb Z+2\pi\mathrm i\mathbb Z).
\]
The smallest possible degree of this (nonconstant) map is two.  We therefore ask for two simple poles at
$[\tau_{\rm lat}\ell+\mathrm i\pi/2]$ and
$[\tau_{\rm lat}\ell+3\mathrm i\pi/2]$; evenness forces
$\tau_{\rm lat}=0$ or $\frac12$ modulo $1$.

There is a unique measure that produces such a function $H$.  Periodicity makes the
residues equal (up to sign in the $\mathrm i\pi$ antiperiod).  This requires weights
proportional to $\operatorname{sech}((j+\tau_{\rm lat})\ell)$.
With that preparation, we are ready for the following one-parameter families of solutions to $\mathcal{U}_\star(s)$ in Definition~\ref{def:binary-measure-variational}.
\begin{definition}
Define the binary
``lattice measure'', for any $0\le\tau_{\rm lat}<1$, by
\begin{equation}\label{eq:upper-binary-lattice}
 \mathfrak m^{\mathrm B}_{\ell,\tau_{\rm lat}}
 =\frac1{Z_{\ell,\tau_{\rm lat}}}\sum_{j\in\mathbb Z}
 \operatorname{sech}((j+\tau_{\rm lat})\ell)
 \delta_{\tanh^2((j+\tau_{\rm lat})\ell)},\qquad
 Z_{\ell,\tau_{\rm lat}}=\sum_{j\in\mathbb Z}
 \operatorname{sech}((j+\tau_{\rm lat})\ell).
\end{equation}
\end{definition}
We also define a spherical measure arising from $\mathfrak{m}^{\mathrm B}_{\ell,1/2}$ by reweighting and pushing forward under the M\"obius map
\[
 M(x)=\frac{1-x}{4x},\qquad
 M(z)-M(x)=\frac{x-z}{4zx}.
\]
This map reverses order on $(0,1)$ and preserves interlacing.  For an
atomic probability measure $\mathfrak m=\sum_jp_j\delta_{x_j}$ with
$x_j>0$ and $A=\sum_jp_j/x_j<\infty$, put
\begin{equation}\label{eq:upper-mobius-transfer}
 \widetilde{\mathfrak m}=\sum_j\frac{p_j}{Ax_j}\delta_{M(x_j)},
 \qquad
 G_{\widetilde{\mathfrak m}}(M(z))=-\frac{4z}{A}G_{\mathfrak m}(z).
\end{equation}
Applying this to $\mathfrak m^{\mathrm B}_{\ell,1/2}$, with
$M(\tanh^2\xi)=(4\sinh^2\xi)^{-1}$, gives
\begin{equation}\label{eq:upper-spherical-lattice}
 \begin{gathered}
 \mathfrak m^{\mathrm S}_{\ell,1/2}=\sum_{j\ge0}w_j\delta_{x_j},\qquad
 \xi_j=(j+\tfrac12)\ell,\qquad x_j=\frac1{4\sinh^2\xi_j},\\
 w_j=\frac{\cosh\xi_j}{\mathcal N_\ell\sinh^2\xi_j},\qquad
 \mathcal N_\ell=\sum_{j\ge0}\frac{\cosh\xi_j}{\sinh^2\xi_j}.
 \end{gathered}
\end{equation}

\begin{lemma}[Evaluation of the lattice measures]\label{lem:upper-lattice}
For $0<\kappa<1$, let $\ell=\pi K'/K$ and $d_\ell=2K/\pi$.
The probability measures in~\eqref{eq:upper-binary-lattice}
and~\eqref{eq:upper-spherical-lattice} have normalization constants
\begin{equation}\label{eq:upper-lattice-normalizations}
 Z_{\ell,0}=d_\ell,\qquad
 Z_{\ell,1/2}=d_\ell\kappa,\qquad
 \mathcal N_\ell=\frac{d_\ell^2\kappa}{2}.
\end{equation}
For $r\in\mathbb R$, the two binary transforms are
\begin{equation}\label{eq:upper-lattice-transforms}
 \begin{aligned}
 H_{\mathfrak m^{\mathrm B}_{\ell,0}}(r)
   &=\operatorname{dn}(d_\ell r,\kappa'),\\
 H_{\mathfrak m^{\mathrm B}_{\ell,1/2}}(r)
   &=\operatorname{dn}(d_\ell r,\kappa')^{-1}.
 \end{aligned}
\end{equation}
Their correlations and costs are
\begin{align}
 S_{\mathrm B}(\mathfrak m^{\mathrm B}_{\ell,0})&=\frac{2E}\pi,
 &C_{\mathrm B}(\mathfrak m^{\mathrm B}_{\ell,0})&=\frac{\ln2}2+B(\kappa),
 \label{eq:upper-binary-plus-statistics}\\
 S_{\mathrm B}(\mathfrak m^{\mathrm B}_{\ell,1/2})
 &=\frac{2(E-\kappa'^2K)}{\pi\kappa},
 &C_{\mathrm B}(\mathfrak m^{\mathrm B}_{\ell,1/2})&=\frac{\ln2}2-B(\kappa),
 \label{eq:upper-binary-minus-statistics}\\
 S_{\mathrm S}(\mathfrak m^{\mathrm S}_{\ell,1/2})&=\frac{K-E}{\kappa K},
 &C_{\mathrm S}(\mathfrak m^{\mathrm S}_{\ell,1/2})&=R_{\mathrm S}(\kappa).
 \label{eq:upper-spherical-statistics}
\end{align}
\end{lemma}
\begin{proof}
Write $d=2K/\pi$ and $q=e^{-\ell}$.  Poisson summation and Jacobi's
Fourier series~\cite[Eq.~22.11.3]{DLMF} give
\[
 F(r):=\sum_{j\in\mathbb Z}\operatorname{sech}(r-j\ell)
       =d\operatorname{dn}(dr,\kappa').
\]
The half-period identity is
\[
 \operatorname{dn}(u+K',\kappa')
 =\frac{\kappa}{\operatorname{dn}(u,\kappa')},\qquad
 F(r+\ell/2)F(r)=d^2\kappa.
\]
Pairing the terms at $a$ and $-a$ using~\eqref{eq:upper-binary-kernel}
gives~\eqref{eq:upper-lattice-transforms} and the normalizing factors
in~\eqref{eq:upper-lattice-normalizations}.
\cite[Eq.~22.11.14]{DLMF} gives
\[
 \sum_{j\in\mathbb Z}\operatorname{sech}^2(r-j\ell)
 =d^2\left(\operatorname{dn}^2(dr,\kappa')+\frac EK-1\right).
\]
Evaluation at $r=0$ and $r=\ell/2$ gives the binary correlations.
Evaluation at $r=\ell/2+\mathrm i\pi/2$ gives
\[
 2\sum_{j\ge0}\operatorname{csch}^2\xi_j=d^2(1-E/K).
\]
The half-period identity gives the derivative
\[
 F'(\ell/2+\mathrm i\pi/2)
 =2\mathrm i\sum_{j\ge0}\frac{\cosh\xi_j}{\sinh^2\xi_j}
 =\mathrm i d^2\kappa,
\]
using the fact that $-\mathrm i$ is the residue of $F$ at $\mathrm i\pi/2$.  Thus $\mathcal N_\ell=d^2\kappa/2$, and the correlation is
\[
 S_{\mathrm S}(\mathfrak m^{\mathrm S}_{\ell,1/2})
 =\frac1{\mathcal N_\ell}\sum_{j\ge0}\operatorname{csch}^2\xi_j
 =\frac{K-E}{\kappa K}.
\]

For the costs, write $C_+$ and $C_-$ for the binary costs at shifts
$0$ and $\frac12$, and $C_{\mathrm S}$ for the spherical cost.
To evaluate the costs, we must find the $y_j$ interlacing the $x_i$,
which are the zeros, resp.~the poles, of the Cauchy transform $G_{\mathfrak m}(z)$.
The half-period identity $F(r+\ell/2)F(r)=d^2\kappa$ places the zeros
of $F$ halfway between its poles.  The two binary transforms are
$F/d$ and $d/F$, so the poles associated to $+$ are the zeroes associated to $-$ and vice versa.
The spherical zeros are the images of the half-shifted binary zeros under $M$
by~\eqref{eq:upper-mobius-transfer}.

The cost identity~\eqref{eq:upper-matrix-statistics} extends to these
infinite interlacing pairs and gives the three entropy sums
\begin{align}
 C_+&=\sum_{n\ge0}f_{\mathrm B}(\tanh^2(n\ell))
       -\sum_{n\ge0}f_{\mathrm B}(\tanh^2((n+\tfrac12)\ell)),
 \label{eq:upper-raw-binary-plus-cost}\\
 C_-&=\sum_{n\ge0}f_{\mathrm B}(\tanh^2((n+\tfrac12)\ell))
       -\sum_{n\ge1}f_{\mathrm B}(\tanh^2(n\ell)),
 \label{eq:upper-raw-binary-minus-cost}\\
 C_{\mathrm S}&=\sum_{n\ge0}f_{\mathrm S}((4\sinh^2((n+\tfrac12)\ell))^{-1})
       -\sum_{n\ge1}f_{\mathrm S}((4\sinh^2(n\ell))^{-1}).
 \label{eq:upper-raw-spherical-cost}
\end{align}
All sums are absolutely convergent by the exponential decay of the
scalar entropies.  Applying the finite cost identity to finite atomic
approximations and using weak continuity
from Lemma~\ref{lem:upper-continuity} identifies them with the measure costs.

It remains to evaluate these sums.  The entropy identities are
\[
 f_{\mathrm B}(\tanh^2a)=(1-a\partial_a)\ln(1+e^{-2a}),\qquad
 f_{\mathrm S}((4\sinh^2a)^{-1})=-(1-a\partial_a)\ln(1-e^{-2a}).
\]
These are the fermionic and bosonic single-mode entropies at energy
$2a$.  On a lattice $a=c\ell$, define the differential operator
$D=1-\ell\partial_\ell$.  The difference of the integer and half-integer lattice sums
is $D$ applied to logarithms of partition-function ratios:
\[
 P_+=\prod_{n\ge1}\frac{1+q^{2n}}{1+q^{2n-1}},\qquad
 P_{\mathrm S}=\prod_{n\ge1}\frac{1-q^{2n}}{1-q^{2n-1}}.
\]
Explicitly,
\begin{align*}
 C_+&=\ln2+D\ln P_+,\qquad C_-=\ln2-C_+,\\
 C_{\mathrm S}&=D\ln P_{\mathrm S}.
\end{align*}
We remark that the products $P_+$ and $P_{\mathrm S}$ converge locally uniformly for $\ell>0$, so the differentiation is well-defined.

Write $\theta_j=\theta_j(0,q)$ for the values of Jacobi theta functions at
$q=e^{-\ell}$.  Their product formulas~\cite[Eqs.~20.4.3--20.4.4]{DLMF} are
\[
 \theta_2=2q^{1/4}\prod_{n\ge1}(1-q^{2n})(1+q^{2n})^2,
 \qquad
 \theta_3=\prod_{n\ge1}(1-q^{2n})(1+q^{2n-1})^2.
\]
Taking their ratio and product, respectively, gives
\[
 P_+^2=\frac{\theta_2}{2q^{1/4}\theta_3},\qquad
 P_{\mathrm S}^2=\frac{\theta_2\theta_3}{2q^{1/4}}.
\]
The theta expressions for the elliptic modulus and
period~\cite[Eqs.~20.9.1--20.9.2]{DLMF} are
\[
 \kappa=\frac{\theta_2^2}{\theta_3^2},\qquad
 K=\frac\pi2\theta_3^2.
\]
Substituting the two product ratios therefore gives
\[
 \kappa=4q^{1/2}P_+^4,\qquad
 \frac{K\sqrt\kappa}{\pi}=q^{1/4}P_{\mathrm S}^2.
\]
Since $D\ln q=0$, it follows that
\[
 C_+=\frac12\ln2+\frac14D\ln\kappa,\qquad
 C_{\mathrm S}=\frac12D\ln\frac{K\sqrt\kappa}{\pi}.
\]
Finally, we take derivatives~\cite[\S19.4(i)]{DLMF}
\[
 \frac{d\ell}{d\kappa}=-\frac{\pi^2}{2\kappa\kappa'^2K^2},\qquad
 \frac{dK}{d\kappa}=\frac{E}{\kappa\kappa'^2}-\frac K\kappa,
\]
and hence
\begin{align*}
 C_\pm&=\frac12\ln2
       \pm\left(\frac14\ln\kappa+\frac{\kappa'^2KK'}{2\pi}\right)
       =\frac12\ln2\pm B(\kappa),\\
 C_{\mathrm S}&=\frac12\ln\frac{K\sqrt\kappa}{\pi}
       +\frac{(2E-\kappa'^2K)K'}{2\pi}=R_{\mathrm S}(\kappa).
\end{align*}
The last equality follows from Legendre's relation.
\end{proof}

Since $s=1-2\delta$, the binary evaluations in
Lemma~\ref{lem:upper-lattice} are the two formulas
\eqref{eq:binary-profile}; the spherical evaluation is
\eqref{eq:spherical-profile}.  As $\ell\downarrow0$, both binary
measures tend to the arcsine law on $[0,1]$, with correlation $2/\pi$
and cost $(\ln2)/2$.

\begin{proof}[Proof of Theorem~\ref{thm:main}]
Use correlation as the common parameter, and put
$F_{\mathrm B}(s)=\ln2\,\mathcal J((1-s)/2)$.
Theorem~\ref{thm:elliptic-obstacle} identifies the contact threshold
at rate $F_\star(s)$ with $s$.
For $0<s'<s$, Theorem~\ref{thm:barrier} therefore bounds the lower
asymptotic rate at correlation $s$ from below by $F_\star(s')$.
Letting $s'\uparrow s$ gives the lower bound $F_\star(s)$.
For $s<s''<1$, choose the lattice measure of correlation $s''$.
Lemmas~\ref{lem:upper-approximation} and~\ref{lem:upper-lattice}
give an upper bound $F_\star(s'')$.  Letting $s''\downarrow s$
gives the reverse inequality.  The rate functions are continuous, including the
``half-rate point'' $s=2/\pi$: both branches tend to $(\ln2)/2$ because
$B(\kappa)\to0$ as $\kappa\uparrow1$.
Thus the liminf and limsup agree.  Dividing the binary equality by
$\ln2$ proves the theorem in bits.
\end{proof}

\begin{corollary}[Optimality in the measure problem]
For $0<s<1$, $\mathcal U_\star(s)=F_\star(s)$.
The displayed lattice measures attain this value; at distance $1/2-1/\pi$,
a minimizing measure is the arcsine law
$dx/(\pi\sqrt{x(1-x)})$ on $[0,1]$.
\end{corollary}
\begin{proof}
The approximation lemma and Theorem~\ref{thm:main} give
$F_\star(s)\le\mathcal U_\star(s)$, and the lattice evaluation gives
the opposite inequality.  At the junction, the arcsine law has
$g(u)=(1-u^2)^{-1/2}$, correlation $2/\pi$, and cost $(\ln2)/2$.
\end{proof}

\section*{Acknowledgements}
We thank Alex Samorodnitsky, Alexander Barg, and Henry Cohn for helpful comments.

\bibliographystyle{alpha}
\begingroup
\raggedright
\bibliography{references}
\endgroup

\end{document}